\documentclass[11pt]{article}
\usepackage{bcc24}
\usepackage{hyperref}

\usepackage{tikz}

\def\nfrac#1#2{{\textstyle\frac{#1}{#2}}}
\def\dfrac#1#2{\lower0.15ex\hbox{\large$\frac{#1}{#2}$}}

\def\svec{\boldsymbol{s}}
\def\tvec{\boldsymbol{t}}
\def\kvec{\boldsymbol{k}}
\def\wvec{\boldsymbol{w}}
\def\xvec{\boldsymbol{x}}
\def\kmax{k_{\max}}
\def\G{\mathcal{G}}

\newcommand{\doi}[1]{\href{http://dx.doi.org/#1}{\texttt{doi:#1}}}

\bcctitle{Generating graphs randomly\footnote{This survey has been published as a chapter in \emph{Surveys in Combinatorics 2021} (Dabrowski et al., eds), \doi{10.1017/9781009036214}. This version is free to view and download for private research and study only. Not for re-distribution or re-use. \copyright\ Cambridge University Press 2021}}
\bccshorttitle{Generating graphs randomly}

\bccname{Catherine Greenhill}
\bccshortname{Catherine Greenhill}

\bccaddressa{School of Mathematics and Statistics\\ UNSW Sydney\\ Sydney NSW 2052, Australia}
\bccemaila{c.greenhill@unsw.edu.au}

\begin{document}
\makebcctitle

\begin{abstract}
Graphs are used in many disciplines to model the relationships that exist
between objects in a complex discrete system.  Researchers may wish to
compare a network of interest to a ``typical'' graph from a family (or ensemble) 
of graphs which are similar in some way.
One way to do this is to take a sample of several random graphs from the family, to
gather information about what is ``typical''.   Hence there is a need for
algorithms which can generate graphs uniformly (or approximately uniformly)
at random from the given family.  Since a large sample may be required,
the algorithm should also be computationally efficient.

Rigorous analysis of such algorithms is
often challenging, involving both combinatorial and probabilistic arguments.
We will focus mainly on the set of all simple graphs with a particular degree sequence,
and describe several different algorithms for sampling graphs from this family
uniformly, or almost uniformly.
\end{abstract}

\section{Introduction}\label{s:introduction}

The modern world is full of networks, and many researchers use
graphs to model real-world networks of interest.  When studying
a particular real-world network it is often convenient to define a
family, or
\emph{ensemble}, of graphs which are similar to the network in some way.  
Then 
a random element of the ensemble provides a \emph{null model} against
which the significance of a particular property of the real-world
model can be tested.  For example, a researcher may observe that
their network contains what looks like a large number of copies
of a particular small subgraph $H$, also called a ``motif'' in network science.
If this number is large compared to the average number of
copies of $H$ in some appropriate ensemble of graphs, then this provides
some evidence that the high frequency of this motif may be related to 
the particular function of the real-world network. 
(For more on network motifs see for example~\cite{network-motifs}.)

In this setting, the null model is a random graph model, and it may be possible
to analyse the relevant properties using probabilistic combinatorics.
Where this is not possible, it is very convenient to have an
algorithm which provides uniformly random (or ``nearly'' uniformly
random) graphs from the ensemble, so that the average number of
copies of $H$ can be estimated empirically.  Such an algorithm should also be
efficient, as a large sample may be needed.
(In this survey, ``efficient'' means ``computationally efficient''.)
Another motivation for the usefulness of algorithms for sampling graphs
can be found in the analysis of algorithms which take graphs as input,
especially when the worst-case complexity bound is suspected to be
far from tight in the average case.  

The aim of this survey is to describe some of the randomized algorithms 
which have
been developed to efficiently sample graphs with certain properties,
with particular focus on the problem of
\emph{uniformly sampling graphs with a given degree sequence}.
We want to understand how
close the output distribution is to uniform, and how the runtime of the
algorithm depends on the number of vertices.  Hence we restrict our attention to algorithms
which have
been rigorously analysed and have certain performance guarantees.
In particular, statistical models such as the exponential random graph model
(see for example~\cite{CAR,Holland-Leinhardt,LKR,Wasserman-Pattison}),
will not be discussed.
Although we mainly restrict our attention to simple, undirected graphs,
most of the ideas discussed in this survey can also be applied to 
bipartite graphs, directed graphs
and hypergraphs, which are all extremely useful in modelling real-world
networks.  

It is still an open problem to find an efficient algorithm for
sampling graphs with an arbitrary degree sequence.  
For bipartite graphs, however, the sampling problem was solved
for arbitrary degree sequences by Jerrum, Sinclair and Vigoda~\cite{JSV},
as a corollary of their breakthrough work on approximating the permanent.  
See Section~\ref{s:jerrum-sinclair} for more detail.

This is not a survey on random graphs.  
(Our focus is on \emph{how} to randomly sample a graph from some
family efficiently, from an algorithmic perspective,
rather than on the properties of the resulting random graph.)
However, 
some techniques are useful both as tools to analyse random graphs
and as procedures for producing random graphs.  There are many 
texts on random graphs~\cite{bollobas-book, FK-book, JLR},
as well as Wormald's excellent survey on random regular graphs~\cite{wormald1999}.

Before proceeding, we remark that in network science, the phrase ``graph 
sampling algorithm'' can refer to an algorithm
for sampling vertices or subgraphs \emph{within} a given (huge) graph
(see for example~\cite{graph-sampling-algorithms}).
For this reason, we will avoid using this phrase and will instead refer to
``algorithms for sampling graphs''.

\section{Preliminaries and Background}\label{s:preliminaries}

\subsection{Notation and assumptions}

Let $[a] = \{ 1,2,\ldots, a\}$ for any positive integer $a$.

A \emph{multigraph} $G = (V,E)$ consists of a set of vertices $V$ and
a multiset $E$ of edges, where each edge is an unordered pair of vertices
(which are not necessarily distinct).  A loop is an edge of the form $\{v,v\}$
and an edge is repeated if it has multiplicity greater than one.
A \emph{graph} is a \emph{simple} multigraph: that is, a multigraph with 
no loops and no repeated edges.  
All graphs are finite and labelled,
so $V$ is a finite set of distinguishable vertices.
A directed multigraph is defined similarly, except that edges are now ordered pairs.
A directed graph is a directed multigraph which is simple, which means that it has no 
(directed) loops and no repeated (directed) edges.

Throughout, $n$ will be the number of vertices of a graph, unless otherwise
specified.  We usually assume that the vertex set is $[n]$.

Standard asymptotic notation will be used, and asymptotics are as $n\to\infty$
unless otherwise specified.   Let $f$, $g$ be real-valued functions of~$n$. 
\begin{itemize}
\item Write $f(n)=o(g(n))$ if $\lim_{n\to\infty} f(n)/g(n)=0$.
\item Suppose that $g(n)$ is positive when $n$ is sufficiently large.
We write $f(n) = O(g(n))$ if there exists a constant $C$ such that
$|f(n)| \leq C\, g(n)$ for all $n$ sufficiently large.
\item Now suppose that $f(n)$ and $g(n)$ are both positive when $n$ is 
sufficiently large.  If $f=O(g)$ and $g=O(f)$ then we 
write $f(n) = \Theta(g(n))$.
\end{itemize}
We sometimes write $\approx$ to denote
an informal notion of ``approximately equal''.
In pseudocode, we write ``u.a.r.'' as an abbreviation for
\emph{uniformly at random}.  

When calculating the runtime of algorithms, we use the ``Word RAM'' model of
computation~\cite{FW,hagerup}. In this model, elementary operations on integers with
$O(\log n)$ bits can be performed in unit time. 

Randomised algorithms require a source of randomness.
We assume that we have a perfect generator
for random integers uniformly distributed in $\{ 1,2,\ldots, N\}$
for any positive integer $N$. Furthermore, we assume that this perfect
generator takes unit time whenever $N$ has $O(\log n)$ bits.

\subsection{Which graph families?}\label{s:graph-families}

It is very easy to sample from some graph families:
\begin{itemize}
\item 
Let $\mathcal{S}(n)$ denote the set of all $2^{\binom{n}{2}}$ 
graphs on the vertex set $[n]$.    We can sample from $\G(n)$ very easily:
flip a fair coin independently for each unordered pair of distinct vertices 
$\{j,k\}$, and add
$\{j,k\}$ to the edge set if and only if the corresponding coin flip comes
up heads.
Every graph on $n$ vertices is equally likely, so this gives an exactly
uniform sampling algorithm with runtime $O(n^2)$.  
\item
Next we might consider $\mathcal{S}(n,m)$, the set
of all $\binom{\binom{n}{2}}{m}$ graphs on the vertex set $[n]$ with precisely
$m$ edges.  A uniformly random graph from this set can be generated 
edge-by-edge, starting with the vertex set $n$ and no edges. At each step,
choose a random unordered pair of distinct vertices $\{j,k\}$, without replacement,
and add this edge to the graph.  When the graph has $m$ edges, it is
a uniformly random element of $\mathcal{S}(n,m)$.   This algorithm has
runtime $O(n^2)$.
Letting $G_i$ denote the graph obtained after $i$ edges have been added,
the sequence $G_0,G_1,\ldots, G_m$ is known as the \emph{random graph process},
with $G_i$ a uniformly-random element of $\mathcal{S}(n,i)$ for all $i\in [m]$.
\end{itemize}

A uniform element from $\mathcal{S}(n,m)$ corresponds to the 
\emph{Erd\H{o}s--R{\' e}nyi random graph} $\mathcal{G}(n,m)$,
while the \emph{binomial random graph model} $\G(n,p)$ is obtained by 
adapting the process for sampling from $\mathcal{S}(n)$ described above, replacing the
fair coin by a biased coin which comes up heads with probability $p$.
These two random graph models have been
the subject of intense study for more than 60 years, 
see for example~\cite{bollobas-book,ER59,FK-book,gilbert,JLR}.
However, since polynomial-time sampling is easy for both
of these families (as described above), we will say no more about them.

\bigskip

Instead, our focus will be on algorithms for sampling graphs with a given
degree sequence. 
More generally, we might be interested in bipartite graphs, directed graphs or
hypergraphs with a given degree sequence.  Alternatively, we may want
to sample graphs with a given degree sequence and
some other property, such as connectedness or triangle-freeness.
There are many variations, but our main focus will be on sampling
from the set $G(\kvec)$ defined below. 

\begin{Def}
\label{def:Gd}
A graph $G$ on vertex set $[n]$ has \emph{degree sequence} 
$\kvec = (k_1,\ldots, k_n)$ if $\operatorname{deg}_G(j) = k_j$ for all $j\in [n]$, 
where $\operatorname{deg}_G(j)$ denotes the
degree of $j$ in $G$.  Let $G(\kvec)$ denote the set of all graphs with 
degree sequence $\kvec$.  
A sequence $\kvec=(k_1,\ldots, k_n)$ of nonnegative integers with even
sum is \emph{graphical} if  $G(\kvec)$ is nonempty.
A graph with degree sequence $\kvec$ is a \emph{realization} of $\kvec$.
\end{Def}

We do not assume here that the elements of $\kvec$ are in non-ascending
order, though we will usually assume that all entries of $\kvec$ are positive.
Unlike the binomial random graph model $\G(n,p)$, the edges of a
randomly-chosen element of $G(\kvec)$ are not independent.
This lack of independence makes sampling from $G(\kvec)$ a non-trivial task.

If $\kvec = (k,k,\ldots, k)$ has every entry equal to $k$, then we
say that $\kvec$ is \emph{regular}.  The set of all $k$-regular graphs on
the vertex set $[n]$ will be denoted by
$G(n,k)$ instead of $G(\kvec)$.

\bigskip

We close this subsection with some more comments on graphical degree sequences.
The characterisations of 
Erd\H{o}s and Gallai~\cite{EG} and Havel and Hakimi~\cite{havel,hakimi} 
both give algorithms which can be used to decide,
in polynomial time, whether a given sequence is graphical. 
The Erd\H{o}s--Gallai Theorem says that if $k_1\geq \cdots \geq k_n$ then
$\kvec$ is graphical if and only if $\sum_{j=1}^n k_j$ is even and
\[ \sum_{j=1}^p k_j \leq p(p-1) + \sum_{j=p+1}^n \min\{ k_j,\, p\}\]
for all $p\in [n]$.
To avoid trivialities, we will always assume that the sequence $\kvec$ is graphical.
The Havel--Hakimi characterisation also assumes that entries of $\kvec$ are
in non-decreasing order, and states that $\kvec$ is graphical if and only if
\[ (k_2-1,\ldots, k_{k_1+1}-1,k_{k_1+2},\ldots, k_n)\]
has no negative entries and is graphical.  This leads to a greedy
algorithm to construct a realisation
of $\kvec$ in runtime $O(n^2)$: join vertex~1 to each of vertices 2,\ldots, $k_1+1$,
delete vertex~1, reduce the target degree of vertices $2,\ldots, k_1+1$
by~1, sort the new degree sequence into nonincreasing order if necessary,
and recurse.  The runtime of this greedy algorithm is $O(n^2)$.

\subsection{What kind of sampling algorithm?}

To be more precise about our goals, we need some definitions.
Let $(\Omega_n)_{n\in\mathcal{I}}$ be a sequence of finite sets indexed by a 
parameter $n$ from some infinite index set $\mathcal{I}$, such as
$\mathcal{I}=\Z^+$ or $\mathcal{I}=2\Z^+$.
Asymptotics are as $n$ tends to infinity along elements of $\mathcal{I}$.
We assume that $|\Omega_n|\to\infty$ as $n\to\infty$.

The reason that we consider a sequence of sets, rather than just one
set, is that it makes no sense to say that the runtime of an algorithm
is polynomial for a particular set $\Omega$.  If the runtime of an
algorithm for sampling from $\Omega$ is $T$,
then we could say that this is a \emph{constant-time algorithm}
with constant $T$, but then we learn nothing about how long the algorithm
might take when given a different set as input.
Having said that, in our notation we often drop the sequence notation
and simply refer to $\Omega_n$.

As a general rule, we say that a (uniform) sampling algorithm for $\Omega_n$ is
\emph{efficient} if its runtime is bounded above by a polynomial in
$\log(|\Omega_n|)$, as it takes $\log(|\Omega_n|)$ bits to describe an
element of $\Omega_n$.  However, the runtime of the algorithm may
have a deterministic upper bound, or it may be a random variable,
leading to the following paradigms:
\begin{itemize}
\item A \emph{Monte Carlo} algorithm is a randomised algorithm which is
guaranteed to terminate after
some given number of steps, but has some probability of incorrect output.
(The probability of incorrect output should be small.)
A~Monte Carlo sampling algorithm for $\Omega_n$ is efficient if
its runtime is bounded above by a polynomial in $\log(|\Omega_n|)$.
\item A \emph{Las Vegas} algorithm is a randomised algorithm which
is guaranteed to provide correct output with probability~1, but may have
no deterministic upper bound on its running time.  A Las Vegas
sampling algorithm for $\Omega_n$ is efficient
if its \emph{expected} runtime is bounded above by a polynomial in
$\log(|\Omega_n|)$.
\end{itemize}

Next, we focus on the output of the algorithm, and give three different definitions
of ``close to uniform''.  
First we need a notion of distance for probability distributions.

\begin{Def}
\label{def:dTV}
Let $\sigma$ and $\pi$ be two probability distributions on the finite
set $\Omega$.  The \emph{total variation distance} between $\sigma$ and
$\pi$, denoted $d_{TV}(\sigma,\pi)$, is given by
\[
    d_{TV}(\sigma,\pi) 
    = \dfrac{1}{2}\sum_{x\in\Omega } |\sigma(x) - \pi(x)|
    = \max_{ S\subseteq \Omega}\, |\sigma(S) - \pi(S)|.
\]
Here $\sigma(S) = \sum_{x\in S} \sigma(x)$ for any event $S\subseteq \Omega$,
and similarly for $\pi(S)$.
\end{Def}

Suppose that some sampling algorithm over $\Omega_n$ has output distribution
$\sigma_n$, and let $\pi_n$ denote the uniform distribution over $\Omega_n$,
for any $n\in\mathcal{I}$.  
\begin{itemize}
\item
If $\sigma_n=\pi_n$ for all $n\in\mathcal{I}$ then we say that the algorithm 
is a \emph{uniform sampling algorithm} or \emph{uniform sampler}.
\item
If $\lim_{n\to\infty}\, d_{TV}(\sigma,\pi) = 0$ then we say that
the algorithm is an \emph{asymptotically uniform sampler}.
In this situation it is usually not possible to increase the accuracy
by running the algorithm for longer, as the total variation distance
depends only on $n$.
\item If $d_{TV}(\sigma_n,\pi_n) < \varepsilon$ for some positive constant
$\varepsilon$ then we say that the algorithm is an
\emph{almost uniform sampler}, and that the output is $\varepsilon$-\emph{close}
to uniform.  Often, $\varepsilon$ is provided by the user, and higher
accuracy (smaller $\varepsilon$) can obtained at the cost of a longer runtime.
\end{itemize}

The Markov chain approach to sampling, when successful, provides
an algorithm called an FPAUS.
See for example~\cite[Chapter~3]{jerrum-book}.

\begin{Def}
\label{def:FPAUS}
Let $(\Omega_n)_{n\in\mathcal{I}}$ be a sequence of finite sets
indexed by a parameter $n$ from some infinite index set $\mathcal{I}$,
such that $|\Omega_n|\to\infty$ as $n\to\infty$.
A \emph{fully-polynomial almost uniform sampler (FPAUS)} for sampling from
$\Omega_n$ is an algorithm that, with probability at least $\nfrac{3}{4}$,
outputs an element of $\Omega_n$ in time polynomial in $\log|\Omega_n|$ and $\log(1/\varepsilon)$, such that the output distribution is \emph{$\varepsilon$-close} to the 
uniform distribution on $\Omega_n$ in total variation distance.
(That is, $d_{TV}(\sigma_n,\pi_n) < \varepsilon$ where $\sigma_n$ is the
output distribution and $\pi_n$ is the uniform distribution on $\Omega_n$.)
\end{Def}

If $\Omega_n = G(\kvec)$ for some graphical sequence $\kvec = (k_1,\ldots, k_n)$
then $\log |\Omega_n| = O(M\log M)$, 
where $M$ is the sum of the entries of $\kvec$.  This can
be proved using the configuration model: see (\ref{Gk-size}).
So an FPAUS for $G(\kvec)$ must
have running time bounded above by a polynomial in $n$
and $\log(1/\varepsilon)$, since $M\leq n^2$.  

\bigskip

Sampling and counting are closely related, and an algorithm for
one problem can often be transformed into an algorithm for the other.
While our
focus is firmly on sampling, we will also need the following
definition which describes a good approximate counting algorithm. 

\begin{Def}
\label{def:FPRAS}
Let $(\Omega_n)_{n\in\mathcal{I}}$ be a sequence of finite sets
indexed by a parameter $n$ from some infinite index set $\mathcal{I}$,
such that $|\Omega_n|\to\infty$ as $n\to\infty$.
A \emph{fully-polynomial randomised approximation scheme} (FPRAS) for $\Omega_n$
is an algorithm which accepts as
input a parameter $\varepsilon>0$ and outputs
an estimate $X$ for $|\Omega_n|$ such that 
\[ \Pr\Big( (1-\varepsilon)|\Omega_n|\leq X\leq (1+\varepsilon)|\Omega_n|\Big)
  \geq \dfrac{3}{4},\]
with runtime polynomial in $\log|\Omega_n|$ and $\varepsilon^{-1}$. 
\end{Def}

The probability in this definition can be easily increased from $\nfrac{3}{4}$
to $1-\delta$, for any fixed $\delta\in (\nfrac{3}{4},1)$, by obtaining
$O(\log\delta^{-1})$ estimates and taking the median~\cite[Lemma~2.1]{SJ}.  

\bigskip

Before moving on, we say a little more about the connection between sampling and
counting.
Jerrum, Valiant and Vazirani~\cite{JVV} proved that for \emph{self-reducible} problems,
polynomial-time approximate counting is equivalent to polynomial-time almost-uniform sampling.  
Without going into too much detail, a problem is self-reducible if the solutions
for a given instance can be generated recursively using a small number of
smaller instances of the same problem.  For example, consider the set $\mathcal{M}(G)$
of all matchings (of any size) in a graph $G$.
Remove the edges of $G$ one-by-one (in lexicographical order, say) to form the sequence
\[ G = G_m > G_{m-1} > \cdots > G_1 > G_0 = (V,\emptyset).\]
Then $|\mathcal{M}(G_0)|=1$ and hence $|\mathcal{M}(G)| = 
\prod_{j=1}^{m} |\mathcal{M}(G_{j})|/|\mathcal{M}(G_{j-1})|$.  The $j$'th ratio
is the inverse of the probability that $e_j\not\in M$, where $M$ is
a matching chosen uniformly at random from $\mathcal{M}(G_j)$
and $e_j$ is the unique edge in $E(G_j)\setminus E(G_{j-1})$.
If we can sample almost-uniformly from the sets $\mathcal{M}(G_j)$ to 
sufficient accuracy then estimates for these probabilities
can be multiplied together  to provide a good estimate for $1/|\mathcal{M}(G)|$.
In this way, approximate counting can be reduced to almost-uniform sampling.
See~\cite[Chapter~3]{jerrum-book} for full details.

For sampling graphs, the situation is a little more complicated.
Erd{\H o}s et al~\cite{EKMS} showed that the problem of sampling 
graphs, or directed graphs, with
given degrees can be made self-reducible by supplying as input a 
small set of forbidden edges, and sampling from the set of graphs (or bipartite
graphs, or directed graphs) with specified degrees which avoid the forbidden edges.   
Using the fact that directed graphs can be modelled as bipartite graphs which
avoid a given perfect matching, the set of forbidden edges can be taken to be a star 
(for undirected graphs) or the the union of a star and a perfect matching. 
We remark that the exact counting problem (given a graphical sequence
$\kvec$, calculate $|G(\kvec)|$) is not known
to be \#P-complete, and similarly for bipartite or directed variants.

\subsection{Sampling graphs with given degrees: an overview}\label{s:history}

We assume that for all $n$ in some infinite index set $\mathcal{I}$,
we have a graphical degree sequence $\kvec(n) = (k_1(n),\ldots, k_n(n))$.
The length of $\kvec(n)$ is $n$, and the elements of $\kvec(n)$
might themselves be functions of $n$.   Our sequence $(\Omega_n)_{n\in\mathcal{I}}$
of sets (as discussed in the previous subsection) is given by taking
$\Omega_n = G\big(\kvec(n)\big)$.  From now on, we simply write
$\kvec = (k_1,\ldots, k_n)$ for the degree sequence, but we should remember
that in fact we have a \emph{sequence} of degree sequences, indexed by $n$.

The maximum entry in a degree sequence $\kvec$ is denoted $k_{\max}$.
Let $M = \sum_{j\in [n]} k_j$ be the sum of the degrees. Then $m=M/2$
the number of edges of any graph in $G(\kvec)$.  

There are a few different methods for sampling graphs with 
given degrees (restricted to algorithms which can be rigorously
analysed).  We outline the main approaches here, and go into more detail in 
the subsequent sections.  

\begin{itemize}
\item
The \emph{configuration model} was 
introduced by Bollob{\' a}s~\cite{B1980} in 1980, as a convenient
way to calculate the probability of events in random regular graphs.  
The model can be
used as an algorithm for sampling uniformly at random from $G(\kvec)$.
However, the expected runtime is high unless the degrees are very
small: specifically, $k=O(\sqrt{\log n})$ in the regular case.
\item
McKay and Wormald's \emph{switchings-based algorithm}~\cite{MW90} from 1990
performs (exactly) uniform sampling 
from $G(\kvec)$ in expected polynomial time, for a much wider range of 
degrees than the algorithm arising from the configuration model, namely
$k_{\max} =  O(M^{1/4})$.
Gao and Wormald~\cite{GW} and Arman, Gao and Wormald~\cite{AGW} have
extended and improved the McKay--Wormald algorithm, allowing it to
apply to a wider range of degrees and making it more efficient.  
These algorithms are fast, but a little complicated and difficult
to implement. See the end of Section~\ref{s:switching-recent}.
\item
Another approach is to use a \emph{Markov chain} with uniform stationary 
probability over a state space $G'(\kvec)$ which contains $G(\kvec)$.
If the Markov chain converges rapidly to its 
stationary distribution, and each step of the Markov chain can be implemented
efficiently, then this gives an FPAUS for $G'(\kvec)$.  
In 1990, Jerrum and Sinclair~\cite{JS90} described and analysed a Markov chain
which samples from a set $G'(\kvec)$ of graphs with degree sequence close
to $\kvec$.  The Jerrum--Sinclair chain
is efficient only when $G(\kvec)$ forms a sufficiently large fraction of $G'(\kvec)$.
When this condition on $G(\kvec)$ holds, rejection sampling can be used to restrict 
the output of the chain to $G(\kvec)$.
The requirement that $|G(\kvec)|/|G'(\kvec)|$ is sufficiently large gives
rise to a notion of \emph{stability} of degree sequences.
Another well-studied Markov chain, the \emph{switch chain}, has state
space $G(\kvec)$ and thus avoids rejection sampling.
Most proofs in this area use Sinclair's multicommodity flow
method~\cite{sinclair}, and the resulting bounds, when polynomial,
tend to be rather high-degree and are not believed to be tight.
\item
In 1999, Steger and Wormald~\cite{SW99} presented an algorithm for
performing asymptotically uniform sampling for $k$-regular graphs.
Their aim was to provide an algorithm which is both 
genuinely fast (runtime $O(k^2 n)$  when $k$ is a small power
of $n$) and easy to implement. 
While the idea for the algorithm is motivated by the
bounded-degree graph process~\cite{RW}, the algorithm is presented as
a modification of the configuration model.  
We also describe extensions and enhancements by Kim and Vu~\cite{KV}
and Bayati, Kim and Saberi~\cite{BKS}: in particular, Bayati et al.~\cite{BKS}
used sequential importance sampling to provide
an algorithm which is \emph{almost} an FPAUS, when $k_{\max} = O(M^{\frac{1}{4}-\tau})$ for any constant $\tau>0$.
\end{itemize}

The configuration model is described in Section~\ref{s:config},
followed in Section~\ref{s:SW} by the sequential algorithms 
beginning with the work of Steger and Wormald.
Switchings-based algorithms are discussed in Section~\ref{s:switchings}, 
and Markov chain (MCMC) algorithms are presented in Section~\ref{s:markov}.

\section{The configuration model}\label{s:config}

The \emph{configuration model}, introduced by Bollob{\' a}s~\cite{B1980},
is very useful in the analysis of random graphs with given degrees.
It also arose in asymptotic enumeration work of Bender and Canfield~\cite{BC},
and was first explicitly used as an algorithm by Wormald~\cite{wormald1984}.

Given a degree sequence $\kvec$, and recalling that $M=\sum_{j\in [n]} k_j$,
we take $M$ objects,
called \emph{points}, grouped into $n$ \emph{cells}, where the $j$'th
cell contains $k_j$ points.  Each point is labelled, and hence
distinguishable. (You can think of the points corresponding to vertex
$j$ as being labelled by $(j,1),\ldots, (j,k_j)$, say.  But we
will not refer to these labels explicitly.) 

In the network theory literature (for example~\cite{FLNU}), 
points are sometimes called 
\emph{stubs}, or \emph{half-edges}, without the concept of a cell
(so that $k_j$ half-edges emanate from vertex $j$).  

A \emph{configuration}, also called  a \emph{pairing},
is a partition of the $M$ points into $M/2$ pairs.
This is often described as a perfect matching of the $M$ points.
Given a configuration $P$, shrinking each cell to a vertex 
and replacing each pair by an edge gives a multigraph $G(P)$ with
degree sequence $\kvec$.
The multigraph is simple if it has no loops and no repeated edges,
and in this case we also say that $P$ is simple.

\begin{figure}[ht!]
\begin{center}
\begin{tikzpicture}[scale=0.6]
\draw  (-1,1) circle (0.5 and 0.8);
\draw  (-1,3) circle (0.5 and 0.8);
\node [above] at (0.8,4.5) {\large $\ast$};
\node [above] at (4.0,4.5) {\large $\ast\ast$};
\node [below] at (1.0,-0.5) {\large $\ast\ast$};
\draw (6.0,1) circle (0.5 and 0.8);
\draw (6.0,3) circle (0.5 and 0.5);
\draw (1.0,0) circle (1.3 and 0.5);
\draw (4,4) circle (1.3 and 0.5);
\draw (4,0) circle (0.8 and 0.5);
\draw (1.0,4) circle (1.3 and 0.5);
\draw [fill] (-1,0.5) circle (0.1); \draw [fill] (-1,1.5) circle (0.1);
\draw [fill] (-1,2.5) circle (0.1); \draw [fill] (-1,3.5) circle (0.1);
\draw [fill] (6.0,0.5) circle (0.1); \draw [fill] (6.0,1.5) circle (0.1);
\draw [fill] (6.0,3) circle (0.1); 
\draw [fill] (0.0,0) circle (0.1); \draw [fill] (1.0,0) circle (0.1);
\draw [fill] (2.0,0) circle (0.1); 
\draw [fill] (0.0,4) circle (0.1); \draw [fill] (1.0,4) circle (0.1);
\draw [fill] (2.0,4) circle (0.1); 
\draw [fill] (3.0,4) circle (0.1); \draw [fill] (4.0,4) circle (0.1);
\draw [fill] (5.0,4) circle (0.1); 
\draw [fill] (3.5,0) circle (0.1); \draw [fill] (4.5,0) circle (0.1); 
\draw [-] (-1,0.5) -- (0,0) (1,0) -- (5,4)  (2,0) -- (3,4);
\draw [-] (-1,2.5) -- (6,3) (4.5,0) -- (6,0.5) (6,1.5) -- (3.5,0) (-1,1.5) -- (4,4); 
\draw [-] (-1,3.5) -- (0,4); \draw [-] plot [smooth] coordinates {(1,4) (1.3,4.8) (1.7,4.8) (2,4)};
\begin{scope}[shift={(10,0)}]
\draw (-1,1) circle (0.5 and 0.8);
\draw (-1,3) circle (0.5 and 0.8);
\draw (6.0,1) circle (0.5 and 0.8);
\draw (6.0,3) circle (0.5 and 0.5);
\draw (1.0,0) circle (1.3 and 0.5);
\draw (4,4) circle (1.3 and 0.5);
\draw (4,0) circle (0.8 and 0.5);
\draw (1.0,4) circle (1.3 and 0.5);
\draw [fill] (-1,0.5) circle (0.1); \draw [fill] (-1,1.5) circle (0.1);
\draw [fill] (-1,2.5) circle (0.1); \draw [fill] (-1,3.5) circle (0.1);
\draw [fill] (6.0,0.5) circle (0.1); \draw [fill] (6.0,1.5) circle (0.1);
\draw [fill] (6.0,3) circle (0.1); 
\draw [fill] (0.0,0) circle (0.1); \draw [fill] (1.0,0) circle (0.1);
\draw [fill] (2.0,0) circle (0.1); 
\draw [fill] (0.0,4) circle (0.1); \draw [fill] (1.0,4) circle (0.1);
\draw [fill] (2.0,4) circle (0.1); 
\draw [fill] (3.0,4) circle (0.1); \draw [fill] (4.0,4) circle (0.1);
\draw [fill] (5.0,4) circle (0.1); 
\draw [fill] (3.5,0) circle (0.1); \draw [fill] (4.5,0) circle (0.1); 
\draw [-] (-1,1.5)-- (-1,2.5) (-1,0.5) -- (4,4) (-1,3.5) -- (4.5,0);
\draw [-] (0,4)-- (3.5,0) (1,4) -- (0,0) (2,4) -- (3,4);
\draw [-] (5,4)-- (6,0.5) (6,3)--(1,0) (6,1.5)-- (2,0);
\end{scope}
\end{tikzpicture}
\caption{Two configurations with the same degree sequence}
\label{fig:config-example}
\end{center}
\end{figure}
Figure~\ref{fig:config-example} shows two configurations
with the same degree sequence, namely $\kvec = (3,3,1,2,2,3,2,2)$ if cells
are labelled clockwise from the top-left.  
The small black circles represent points, which are shown inside cells,
and the lines between points represent pairs.
The configuration on the left is not
simple, as it will produce a loop on the vertex corresponding to the cell
marked with ``$\ast$'', and a repeated edge between the vertices
corresponding to the two cells marked with ``$\ast\ast$''.  The configuration on 
the right is simple.

Let $\mathcal{P}(\kvec)$ be the set of all configurations corresponding
to the degree sequence $\kvec$.   The term \emph{configuration model}
typically refers to the uniform probability model over the set $\mathcal{P}(\kvec)$.
A uniformly random configuration from $\mathcal{P}(\kvec)$ 
can be chosen in $O(M)$ time, as follows.  Starting with all
points unmatched, at each step take an arbitrary point $p$ and pair it with
a point chosen uniformly at random from the remaining unmatched points
(excluding $p$).
Once all points have been paired up, we have a configuration $P$
and each configuration is equally likely.  

The configuration model can be used as an algorithm for sampling uniformly 
from $G(\kvec)$, by repeatedly sampling $P\in \mathcal{P}(\kvec)$ uniformly
at random until $G(P)$ is simple.
This algorithm is displayed in Figure~\ref{fig:config}.
\smallskip

\begin{figure}[ht!]
\begin{center}
\hspace{\dimexpr-\fboxrule-\fboxsep\relax}\fbox{
\begin{minipage}{0.8\textwidth}
\begin{tabbing}
  \textsc{Configuration model sampling algorithm}\\
    X \= \kill
  \> XXXs \= \kill 
  \emph{Input:}\>\> graphical sequence $\kvec$\\
  \emph{Output:}\>\> element of $G(\kvec)$\\
  \\
    repeat\\
    \>  choose $P\in \mathcal{P}(\kvec)$ u.a.r.\ \\
    until $G(P)$ is simple\\
    output $G(P)$
\end{tabbing}
\end{minipage}}
\caption{The configuration model as a sampling algorithm}
\label{fig:config}
\end{center}
\end{figure}

Observe that if $G$ is a simple graph with degree sequence $\kvec$ then
$G$ corresponds to exactly $\prod_{j\in [n]} k_j!$ configurations,
as there are $k_j!$ ways to assign points to the edges incident with vertex $j$,
and these assignments can be made independently for each vertex $j\in [n]$.
Hence every element of $G(\kvec)$ is equally likely to be produced
as output of the above process.

This gives a Las Vegas sampling algorithm, with expected runtime
which depends linearly on the probability that a random configuration
is simple.
Hence, the configuration model can be used for efficient sampling when 
the probability that a randomly chosen configuration is simple
is bounded below by $1/p(n)$, for some polynomial $p(n)$.
In this case, the expected number of trials before a simple configuration
is found is at most $p(n)$, and the expected runtime is $O(M\, p(n))$.

A multigraph is simple if and only if it contains no 1-cycles (loops)
and no 2-cycles (arising from repeated edges).   
If the maximum 
degree is not too large compared to the number of edges, then 
in a uniformly random element of $\mathcal{P}(\kvec)$, the
number of 1-cycles and the number of 2-cycles are asymptotically independent
Poisson random variables.
In the $k$-regular case, Bender and Canfield~\cite{BC} proved in 1978
that a uniformly random configuration is simple with probability 
$(1+o(1))\, e^{-(k^2-1)/4}$.  Hence the configuration model for 
$k$-regular graphs gives an expected polynomial time algorithm as long as 
$k = O(\sqrt{\log n})$. 
A very precise estimate of $\Pr(\text{simple})$, with many significant
terms,  was given by McKay and Wormald~\cite{McKW91} in 1991 under the 
assumption that 
$k_{\max}^3 = o(M)$.  To prove the following result, we use 
the estimate (\ref{eq:Pr-simple}) obtained by Janson~\cite{Janson2009}
in 1999, which is valid for a wider range of degrees. 

\begin{theorem} \emph{\cite{Janson2009}}\
Let $R=R(\kvec)$ be defined by $R = \sum_{j\in [n]} k_j^2$.
The configuration model gives a uniform sampling algorithm for $G(\kvec)$.
If $k_{\max}^2 = o(M)$ then 
the expected runtime of this algorithm is
\[ \Theta\big(M\,\exp\big(R^2/(4M^2)\big)\big)\]
when $k_{\max}^2 = o(M)$. 
So the expected runtime is polynomial
if and only if $R = \Theta(M\sqrt{\log n})$.  
In particular, if $k_{\max} = O(\sqrt{\log n})$ then the expected
runtime is polynomial.
\label{thm:config}
\end{theorem}

\begin{proof} (Sketch.)\
The output is distributed uniformly as each element of $G(\kvec)$ is simple,
and hence corresponds to the same number of configurations in $\mathcal{P}(\kvec)$.
The expected number of trials required before a simple configuration is found
is $1/\Pr(\text{simple})$, where $\Pr(\text{simple})$ denotes the
probability that a uniformly chosen configuration from $\mathcal{P}(\kvec)$
is simple.
Janson~\cite{Janson2009} proved that if $k_{\max}^2 = o(M)$ then the 
probability that a random configuration is simple is
\begin{equation}
\label{eq:Pr-simple}
 \Pr(\text{simple}) = \exp\left( - \frac{R^2}{4M^2} + \dfrac{1}{4}\right) + o(1).
\end{equation}
Hence the expected runtime of the algorithm is 
$\Theta\big(M/\Pr(\text{simple})\big)$, which is bounded above by a polynomial
if and only if $R  = \Theta(M\, \sqrt{\log n})$. 
The last statement of the theorem follows since $R\leq k_{\max} M$.
\end{proof}

There are versions of the configuration model which can be used to
sample bipartite graphs, directed graphs or hypergraphs with a given
degree sequence.  In all cases, the expected runtime is polynomial only
for constant or very slowly-growing degrees.

\section{Sequential algorithms and graph processes}\label{s:SW}

The study of graph processes dates back to the very beginnings of the study of
random graphs, in the work of Erd\H{o}s and R{\' e}nyi~\cite{ER60}.
In a random graph process, edges are added to an empty graph one by one, 
chosen randomly from the set of all non-edges, sometimes with additional 
constraints.  
In 1979, Tinhofer~\cite{tinhofer} described such an algorithm for sampling
from $G(\kvec)$ non-uniformly. 
The \emph{a posteriori} output probability could be calculated and,
in theory, this could be combined
with a rejection step in order to achieve uniformly distributed output.
However, the runtime of the resulting algorithm (with the
rejection step) is not known.

Recall that $G(n,k)$ denotes the set of all $k$-regular graphs on $[n]$.
The bounded-degree 
graph process starts with the empty graph on $n$ vertices (with no
edges), and repeatedly chooses two distinct non-adjacent vertices with degree at most
$k-1$, uniformly at random, and joins these two vertices by an edge.
When no such pair of vertices remain, either we have a $k$-regular graph
or the process has become stuck. 
(The name ``bounded-degree graph process'' does not mean that all the degrees
are $O(1)$.  Rather, it means that we add edges sequentially but do not allow
the degree of any vertex to exceed $k$, so we maintain this upper bound on
all degrees.)

 Ruci{\' n}ski and Wormald~\cite{RW}
proved that for any constant $k$, the process produces a $k$-regular
graph with probability $1-o(1)$.  The output distribution
is not uniform, and is not well understood.  
However, it is conjectured that the output of the bounded-degree graph process is
contiguous with the uniform distribution over $G(n,k)$: see Wormald~\cite[Conjecture 6.1]{wormald1999}.  (Two sequences of probability spaces are \emph{contiguous} if any event
with probability which tends to~1 in one sequence must
also tend to~1 in the other.)  

We now turn to sequential algorithms which produce asymptotically uniform
output.

\subsection{The regular case}

Steger and Wormald~\cite{SW99} described an algorithm for
sampling from $G(n,k)$ using the following modification
of the configuration model algorithm.  Instead of choosing a configuration
$P$ uniformly at random, and then rejecting the resulting
graph $G(P)$ if it is not simple, we choose one pair at a time and only
keep those pairs which do not lead
to a loop or a repeated edge.  Specifically, we start with $kn$ points
in $n$ cells, each with $k$ points.  Let $U$ be the set of unpaired points,
which initially contains all $kn$ points.  A set of two points $\{p,p'\}\subseteq U$
is \emph{suitable}
if $p$ and $p'$ belong to different cells, and no pair chosen so far
contains points from the same two cells as $p,p'$.
After repeatedly choosing pairs of suitable points, the algorithm
may get stuck, or else reaches a full configuration $P$ (with $kn/2$ pairs)
and outputs the simple $k$-regular graph $G(P)$.
The algorithm is given in pseudocode in Figure~\ref{fig:SW}.

\begin{figure}[ht!]
\begin{center}
\hspace{\dimexpr-\fboxrule-\fboxsep\relax}\fbox{
\begin{minipage}{0.8\textwidth}
\begin{tabbing}
  \textsc{Steger--Wormald algorithm}\\
    X \= \kill
  \> XXXs \= \kill 
  \emph{Input:}\>\> $n$ and $k$, with $kn$ even\\
  \emph{Output:}\>\> element of $G(n,k)$\\
  \\
  repeat\\
  X \= \kill
  \> let $U$ be the set of all $kn$ points \\
  \> let $P:= \emptyset$\\
    \> repeat\\
    \> X \= \kill
    \> \>  choose a set of two distinct points $\{p,p'\}\subseteq U$ u.a.r. \\
    \> \> if $\{p,p'\}$ is suitable then add $\{p,p'\}$ to $P$ and delete $\{p,p'\}$ from $U$\\
    \> until $U$ contains no suitable pairs of points\\
    until $G(P)$ is $k$-regular\\
    output $G(P)$
\end{tabbing}
\end{minipage}}
\end{center}
\caption{The Steger--Wormald algorithm}
\label{fig:SW}
\end{figure}

Though the explanation above involves the configuration model,
Steger and Wormald state that their algorithm arose from adapting
the bounded-degree processes.  In this setting, the
Steger--Wormald algorithm corresponds to choosing
the vertices $u,v$ to add at the next step with a non-uniform probability.
To be specific, if $k'(x)$ denotes the current degree
of vertex $x$ in the graph formed by the edges chosen so far,  then
the Steger--Wormald algorithm chooses $\{ u,v\}$ as the next edge
with probability proportional to $(k-k'(u))(k-k'(v))$. 

The following theorem is a combination of Steger and Wormald's 
results~\cite[Theorems 2.1, 2.2 and 2.3]{SW99}.

\begin{theorem} \emph{\cite{SW99}}\
\label{thm:SW}
Let $\Pr(G)$ denote the probability that a given graph
$G\in G(n,k)$ is produced as output of the Steger--Wormald algorithm.
\begin{itemize}
\item[\emph{(i)}] If $k=O(n^{1/28})$ then there exists a function $f(n,k)=o(1)$
such that for every $G\in G(n,k)$,
\[ \left|\Pr(G) - |G(n,k)|^{-1}\right| \ < \frac{f(n,k)}{|G(n,k)|}.\]
\item[\emph{(ii)}] If $k=o\big((n/\log^3 n)^{1/11}\big)$ then there exists
a function $f(n,k)=o(1)$ and a subset $\mathcal{X}(n,k)\subseteq G(n,k)$ such that
\[ \Pr(G) = (1 + O\big( f(n,k)\big)\big)\, |G(n,k)|^{-1}\]
for all $G\in\mathcal{X}(n,k)$, and $|\mathcal{X}(n,k)| = (1-f(n,k))\, |G(n,k)|$.
\item[\emph{(iii)}]
Under the same condition as (ii), the expected number of times that the outer loop of the algorithm
is performed (that is, until $G(P)$ is regular) is $1+o(1)$, and hence the
runtime of the algorithm is $O(k^2 n)$.
\end{itemize}
\end{theorem}

In particular, when $k=o\big((n/\log^3 n)^{1/11}\big)$, the output 
distribution of the
Steger--Wormald algorithm is within $o(1)$ of uniform in total variation distance.

\bigskip

Kim and Vu~\cite{KV} gave a new analysis of the Steger--Wormald algorithm
using a concentration result of Vu~\cite{Vu}, increasing the upper bound
on the degree and confirming a conjecture of Wormald~\cite{wormald1999}.

\begin{theorem}\, \emph{\cite{KV}}\
\label{thm:KV}
Let $0<\varepsilon < \dfrac{1}{3}$ be a constant.
Then for any $k\leq n^{1/3 -\varepsilon}$ and $G\in G(n,k)$, 
the probability $\Pr(G)$ that $G$ is output by the Steger--Wormald algorithm 
satisfies 
$\Pr(G) = (1+o(1))\, |G(n,k)|^{-1}$.
\end{theorem}

\subsection{The irregular case, and an almost-FPAUS}\label{s:BKS}

The Steger--Wormald algorithm was generalised to irregular degree
sequences in 2010 by Bayati, Kim and Saberi~\cite{BKS}. 
They stated their algorithm  in terms of graphs, not configurations,
and report failure (rather than restarting) if the procedure gets stuck.
The pseudocode for this algorithm, which is called \textsc{Procedure A} 
in~\cite{BKS}, is given in Figure~\ref{fig:BKS}.
Recall that $m=M/2 = \dfrac{1}{2}\, \sum_{j\in [n]} k_j$.
We write $\binom{[n]}{2}$ for the set of all unordered pairs of distinct 
vertices in $[n]$.

\begin{figure}[ht!]
\begin{center}
\hspace{\dimexpr-\fboxrule-\fboxsep\relax}\fbox{
\begin{minipage}{0.8\textwidth}
\begin{tabbing}
  \textsc{Bayati, Kim and Saberi: procedure A}\\
    X \= \kill
  \> XXXs \= \kill 
  \emph{Input:}\>\> graphical sequence $\kvec$\\
  \emph{Output:}\>\> element of $G(\kvec)$, or \emph{fail}\\
  \\
  repeat\\
   let $E:= \emptyset$  \quad (\emph{set of edges, initially empty})\\
   let $\widehat{\kvec}:= \kvec$ \quad (\emph{current degree deficit})\\
   let $a:= 1$\\
     repeat\\
     X \= \kill
     \>  choose an unordered pair of 
  distinct vertices $\{i,j\}\in \binom{[n]}{2}\setminus E$
  \\ \>\qquad
    with probability proportional to $p_{ij}:=\widehat{k}_i\widehat{k}_j\Big(1 - \frac{k_ik_j}{4m}\Big)$\\
   \>  let $a:= a\times p_{ij}$\\
  \> add $\{i,j\}$ to $E$ and reduce each of $\hat{k}_i$, $\hat{k}_j$ by 1\\
      until no more edges can be added to $E$\\
     if $|E|=m$ then\\
     \> output $G(P)$ and $N = (m!\, a)^{-1}$\\
     else \\
     \> report \emph{fail} and output $N=0$
\end{tabbing}
\end{minipage}}
\end{center}
\caption{Bayati, Kim and Saberi's asymptotically-uniform sampling algorithm}
\label{fig:BKS}
\end{figure}

This procedure is equivalent to the Steger--Wormald algorithm
when $\kvec$ is regular, since then
the factor $1-k_ik_j/(4m)$ does not introduce any bias.
For irregular degrees, this factor is chosen for the following reason.  
If two vertices of high degree are joined by an edge,
then this choice makes it more difficult for the process to complete
successfully.  In~\cite{BKS}, the authors show that
the bias from edge $\{i,j\}$ is roughly $\exp(k_ik_j/(4m))$, and hence
the probability $1-k_ik_j/(4m)\approx \exp(-k_ik_j/(4m))$ is designed
to cancel out this bias.  

Bayati, Kim and Saberi~\cite[Theorem~1 and Theorem~2]{BKS} proved
the following properties of \textsc{Procedure A}.

\begin{theorem} \emph{\cite{BKS}}\
\label{thm:BKS}
Let $\kvec$ be a graphical degree sequence and $\tau>0$ an arbitrary
constant.
\begin{itemize}
\item[\emph{(i)}] Suppose that 
$k_{\max} = O(m^{1/4-\tau})$. 
Then
\textsc{Procedure A} terminates successfully with probability $1-o(1)$
in expected runtime $O(k_{\max}\, m)$, 
and the probability $\Pr(G)$ that any given $G\in G(\kvec)$ is output satisfies 
$\Pr(G) = (1+o(1))\, |G(\kvec)|^{-1}$.
\item[\emph{(ii)}] Now suppose that $\kvec=(k,\ldots, k)$, where $k=O(n^{1/2-\tau})$.
Then \textsc{Procedure A} has output distribution which is within distance $o(1)$
from uniform in total variation distance.  
\end{itemize}
\end{theorem}
Part (i) of this theorem extends the Kim--Vu result (Theorem~\ref{thm:KV}) to the
irregular case with essentially the same condition, since $m=kn$ when
$\kvec$ is $k$-regular.  Similarly, part (ii) of Theorem~\ref{thm:BKS}
generalises Theorem~\ref{thm:SW}(ii) to irregular degree sequences with much higher
maximum degree.  

\bigskip

When successful, \textsc{Procedure A} outputs a graph and a nonnegative 
number $N$.
The value of $N$ is not needed for asymptotically-uniform sampling, but is used
to give a \emph{fully-polynomial randomised approximation scheme} (FPRAS)
for approximating $|G(\kvec)|$, using a technique known as
\emph{sequential importance sampling} (SIS) which we outline below. 
Recall the definition of FPRAS from Definition~\ref{def:FPRAS}.

Let $\mathcal{N}(\kvec)$ be the set obtained by taking all possible
edge-labellings of graphs $G(\kvec)$, labelling the edges $e_1,\ldots, e_m$.
Then $|\mathcal{N}(\kvec)| = m!\, |G(\kvec)|$.  
We can slightly modify \textsc{Procedure A} so that it labels
the edges in the order that they were chosen.  This modified 
\textsc{Procedure A}
produces $H\in \mathcal{N}(\kvec)$ with probability
$P_A(H)$, denoted by $a$ in Figure~\ref{fig:BKS}.  The expected
value of $1/P_A(\cdot)$ for an element of $\mathcal{N}(\kvec)$
chosen according to the distribution $P_A$, is
\[ \sum_{H\in\mathcal{N}(\kvec)} \frac{1}{P_A(H)}\, P_A(H) = |\mathcal{N}(\kvec)| = m!\, |G(\kvec)|.\]
Therefore we can estimate $|G(\kvec)|$ by performing $r$ trials of
\textsc{Procedure~A} and taking the average of the resulting $r$ values
of $\big(m!\, P_A(H_i)\big)^{-1}$.  (Note that $m!\, P_A(H)$ is
precisely the value denoted $N$ in Figure~\ref{fig:BKS} when the
edge-labelled graph $H$ is output.)

Bayati et al. prove~\cite[Theorem 3]{BKS} that 
taking $r=O(\varepsilon^{-2})$ gives an FPRAS for estimating $|G(\kvec)|$.

\bigskip

Finally, 
Bayati et al.~\cite{BKS} showed how to adapt the SIS approach
to estimate $P_A(G)$ for (non-edge-labelled) $G\in G(\kvec)$.  This leads to
an algorithm which is \emph{almost} an FPAUS for $G(\kvec)$,
when $\kmax = O(m^{1/4- \tau})$ for some $\tau >0$.
The algorithm satisfies every condition from the definition of FPAUS
except for the runtime: 
in an FPAUS the runtime must be polynomial in $n$ and $\log(1/\varepsilon)$,
but the algorithm given in~\cite[Section 3]{BKS} has runtime which
is polynomial in $n$ and $1/\varepsilon$.
In a little more detail, the Bayati--Kim--Saberi algorithm proceeds as follows:  
\begin{itemize}
\item Given a graphical degree sequence $\kvec$ and
parameters $\varepsilon, \delta\in (0,1)$, the FPRAS is used to obtain
a sufficiently good estimate $X$ for $|G(\kvec)|$ (with high probability),
and a random graph $G\in G(\kvec)$ is obtained using \textsc{Procedure A}.
\item
Next, we need an estimate $P_G$ for the probability $P_A(G)$ 
that \textsc{Procedure A}
outputs $G$.  This probability is estimated  as follows:
repeatedly choose a random ordering of the edges of $G$, calculate
the probability that these edges were chosen \emph{in this order} during the
execution of \textsc{Procedure A}, and take the average of these probabilities 
(averaged over the different orders chosen).  
\item Finally, $G$ is returned as output of the almost-FPAUS with probability 
given by $\min\{ \frac{1}{c X P_G},\, 1\}$,
where $c$ is a universal constant independent of $\kvec$, $\varepsilon$,
$\delta$.
\end{itemize}

Bayati, Kim and Saberi~\cite[Remark~1]{BKS} 
state their main results can be adapted to give analogous
results for sampling bipartite graphs with given degrees,
under the same assumptions on the maximum degree. 
Independently, Blanchet~\cite{blanchet} used sequential importance sampling 
to give an FPRAS for counting bipartite graphs with given degrees,
when the maximum degree in one part of the vertex bipartition is constant,
while in the other part the maximum degree is $o(M^{1/2})$ and the sum of the
squares of the degrees is $O(M)$.
The arguments provided by
Bayati et al.~\cite{BKS} and Blanchet~\cite{blanchet} 
utilise concentration inequalities and Lyapunov inequalities, 
respectively.

\bigskip

Sequential importance sampling was used by
Chen, Diaconis, Holmes and Liu~\cite{chen} and
Blitzstein and Diaconis~\cite{BD} 
to sample graphs and bipartite graphs with given degrees, but
without fully rigorous analysis.  Sequential importance sampling algorithms
also appear in the physics literature, for example~\cite{DGKTB,KDGBT},
again without rigorous analysis.

While sequential importance sampling algorithms
may perform well in practice in many cases, Bez{\' a}kov{\' a} et al.~\cite{BSSV}
showed that these algorithms are provably slow in some cases.

\subsection{Other graph processes}\label{s:other-processes}

The property of having maximum degree at most $k$ can be rephrased
as the property of having no copy of the star $K_{1,k}$.  More generally,
for a fixed graph $H$,
the $H$-\emph{free process} proceeds from an empty graph by repeatedly
choosing a random edge and adding it to the graph if it does not form a copy
of $H$. See for example~\cite{BR,OT,W}. In particular, 
the \emph{triangle-free process} is very well
studied and has connections with Ramsey Theory~\cite{BK,ESW,PGM} (we do
not attempt to be comprehensive here as there is a large literature on this
topic).  The main focus in this area is extremal, as analysis of these processes
provides a lower bound on the maximum number
of edges possible in an $H$-free graph.  This often involves application of
the differential equations method, see~\cite{wormald-diff}.
Some pseudorandom properties of the output
have been proved for these processes, see for example~\cite{BK,PGM}.
However, it is not clear how far the output distribution varies from uniform,
and so these processes may not be suitable for almost-uniform sampling.

An exception is the work of Bayati, Montanari and Saberi~\cite{BMS},
who adapted the methods of~\cite{BKS} to analyse a sequential algorithm
for generating graphs with a given number of edges and girth greater
than $\ell$ (that is, no cycles of length at most $\ell$), where $\ell$ is a fixed
positive integer.  
As in~\cite{BKS,SW99}, the next edge is chosen non-uniformly, such that
the probability that an edge $e$ is selected is (approximately) proportional
to the number of successful completions of the subgraph $G'\cup\{e\}$,
where $G'$ denotes the current graph.  Bayati et al.~\cite{BMS} prove that the
output of their algorithm is asymptotically uniform after 
$m=O\big(n^{1+1/\left(2\ell(\ell+3)\right)}\big)$ edges have been added. 
The expected runtime of the algorithm is $O(n^2 m)$.

\section{Switchings-based algorithms}\label{s:switchings}

In the sampling algorithm based on the configuration model 
(Figure~\ref{fig:config}), a configuration $P$ is chosen from $\mathcal{P}(\kvec)$ 
uniformly at random, repeatedly, until the corresponding graph $G(P)$ 
is simple.  That is, if $G(P)$ contains any ``defect'' (in this case,
a loop or a repeated edge) then this choice is rejected and we
choose again.  In 1990, McKay and Wormald~\cite{MW90}
introduced a uniform sampling algorithm for $G(\kvec)$ 
which begins by choosing a random element of $\mathcal{P}(\kvec)$ and
rejecting it only if there are ``too many'' defects.
Once a configuration has been found with ``not too many'' defects,
operations called \emph{switchings} 
are applied, one by one, to reduce the number of defects
until a simple configuration is obtained.  To maintain a
uniform distribution, McKay and Wormald
introduce a carefully-chosen rejection probability at each step 
of the process.

We describe the McKay--Wormald algorithm in some detail, as this will
set the scene for the significant improvements introduced by Gao and
Wormald~\cite{GW}, to be discussed in Section~\ref{s:switching-recent}.
The structure of the McKay--Wormald algorithm is based on the
\emph{switching method}, introduced by McKay~\cite{McKay84}.  
The switching method is used in asymptotic enumeration to 
obtain good approximations for the cardinality of large combinatorial
sets, such as the set of all graphs with given degrees~\cite{McKW91},
when the maximum degree is not too large.

To make the phrase ``not too many defects'' precise,
recall that a \emph{loop} in a configuration
is a pair between two points from the same cell.
A \emph{triple pair} in a configuration is a set of three 
distinct non-loop pairs between the same two cells, and a \emph{double pair}
is a set of two distinct non-loop pairs between the same two cells.

\begin{Def}
Let $M_2 = M_2(\kvec) = \sum_{j\in [n]} k_j(k_j-1)$. 
(Note that $M_2$ counts the number of ways
to choose an ordered pair of points from the same cell.)
Define $B_1 = M_2/M$ and $B_2 = (M_2/M)^2$.
Say that 
a configuration $P\in \mathcal{P}(\kvec)$ is \emph{good} if every cell contains
at most one loop, there are no triple pairs, $P$ contains at most
$B_1$ loops and at most $B_2$ double pairs. 
Write $\mathcal{P}^\ast(\kvec)$ for the set of all good configurations
in $\mathcal{P}^\ast$.
\label{def:good-set}
\end{Def}

Combining McKay and Wormald~\cite[Lemma~2 and Lemma~{3$^{\prime}$}]{MW90} 
with~\cite[Lemma~8]{AGW}, we can prove that if $\kmax^4 = O(M)$ then
 there exists a constant $c\in (0,1)$ such that
a uniformly-random element of $\mathcal{P}(\kvec)$ is good with probability
at least $c$.  

Next, let $\mathcal{C}_{\ell,d}$ be the set of all good
configurations with exactly $\ell$ loops and $d$ double pairs.
These sets form a partition of $\mathcal{P}^\ast(\kvec)$.
McKay and Wormald defined two switching operations, which we will
refer to as \emph{loop-switchings} and \emph{double-switchings}.
A loop-switching is used to reduce the number of loops by one, and a 
double-switching is used to reduce the number of double pairs by one.
These switchings are illustrated in Figure~\ref{fig:switchings}.
For example, in the loop-switching, a loop is selected together
with two other pairs, such that there are 5 distinct cells involved,
and performing the switching does not result in any repeated
pairs. The loop-switching transforms an element of $\mathcal{C}_{\ell,d}$
to an element of $\mathcal{C}_{\ell-1,d}$. To describe a loop-switching
we specify an ordered  6-tuple of points, and similarly a double-switching
is specified using an ordered 8-tuple of points.
\begin{figure}[ht!]
\begin{center}
\begin{tikzpicture}[scale=0.6]
\draw (0,0.5) circle (0.5); \draw (0,2) circle (0.5);
\draw (1,3.5) circle (0.8 and 0.5);
\draw (2,0.5) circle (0.5); \draw (2,2) circle (0.5);
\draw [fill] (0,0.5) circle (0.1); \draw [fill] (0,2) circle (0.1);
\draw [fill] (0.75,3.5) circle (0.1); \draw [fill] (1.25,3.5) circle (0.1);
\draw [fill] (2,0.5) circle (0.1); \draw [fill] (2,2) circle (0.1);
\draw [-] (0,0.5) -- (0,2) (2,0.5) -- (2,2);
\draw [-] plot [smooth] coordinates {(0.75,3.5) (0.5,4.3) (1.5,4.3) (1.25,3.5)};
\draw [->, line width=3pt] (3.5,2) -- (4.5,2);
\begin{scope}[shift={(6,0)}]
\draw (0,0.5) circle (0.5); \draw (0,2) circle (0.5);
\draw (1,3.5) circle (0.8 and 0.5);
\draw (2,0.5) circle (0.5); \draw (2,2) circle (0.5);
\draw [fill] (0,0.5) circle (0.1); \draw [fill] (0,2) circle (0.1);
\draw [fill] (0.75,3.5) circle (0.1); \draw [fill] (1.25,3.5) circle (0.1);
\draw [fill] (2,0.5) circle (0.1); \draw [fill] (2,2) circle (0.1);
\draw [-] (0,0.5) --  (2,0.5) (0,2) -- (0.75,3.5)  (1.25,3.5)-- (2,2);
\end{scope}
\begin{scope}[shift={(12,0)}]
\draw (0,0) circle (0.5); \draw (0,4) circle (0.5);
\draw (0,2) circle (0.5 and 0.8); \draw (2,2) circle (0.5 and 0.8);
\draw (2,0) circle (0.5); \draw (2,4) circle (0.5);
\draw [fill] (0,0) circle (0.1); \draw [fill] (0,4) circle (0.1);
\draw [fill] (0,1.75) circle (0.1); \draw [fill] (2,1.75) circle (0.1);
\draw [fill] (0,2.25) circle (0.1); \draw [fill] (2,2.25) circle (0.1);
\draw [fill] (2,0) circle (0.1); \draw [fill] (2,4) circle (0.1);
\draw [-] (0,0) -- (2,0)  (0,1.75) -- (2,1.75) (0,2.25) -- (2,2.25)  (0,4) -- (2,4);
\draw [->, line width=3pt] (3.5,2) -- (4.5,2);
\begin{scope}[shift={(6,0)}]
\draw (0,0) circle (0.5); \draw (0,4) circle (0.5);
\draw (0,2) circle (0.5 and 0.8); \draw (2,2) circle (0.5 and 0.8);
\draw (2,0) circle (0.5); \draw (2,4) circle (0.5);
\draw [fill] (0,0) circle (0.1); \draw [fill] (0,4) circle (0.1);
\draw [fill] (0,1.75) circle (0.1); \draw [fill] (2,1.75) circle (0.1);
\draw [fill] (0,2.25) circle (0.1); \draw [fill] (2,2.25) circle (0.1);
\draw [fill] (2,0) circle (0.1); \draw [fill] (2,4) circle (0.1);
\draw [-] (0,0) -- (0,1.75)  (0,2.25) -- (0,4) (2,2.25) -- (2,4)  (2,0) -- (2,1.75);
\end{scope}
\end{scope}
\end{tikzpicture}
\caption{A loop-switching (left) and a double-switching (right)}
\label{fig:switchings}
\end{center}
\end{figure}

It is possible to remove loops and double pairs using simpler switchings.
In fact, McKay used simpler switching operations 
(as illustrated in Figure~\ref{fig:switchings-Janson} below)
in a very early application~\cite{McKay85} of the switching method for asymptotic 
enumeration.  Subsequently, McKay and Wormald found that by using the slightly
more complicated switchings shown in Figure~\ref{fig:switchings},
they could obtain an asymptotic formula with vanishing error for a wider
range of degree sequences (with a weaker bound on the maximum degree, to
be precise), compared with the result of~\cite{McKay85}. The benefits
obtained by using the slightly more complicated switchings also hold
here in the algorithmic setting.

The first step of McKay and Wormald's algorithm is to repeatedly
choose a uniformly random element of $\mathcal{P}(\kvec)$ until
it is good.  At this point, the configuration $P$ is a uniformly
random element of $\mathcal{P}^\ast(\kvec)$.
Next, if $P$
contains a loop then a loop-switching is chosen uniformly at
random from the set of all available options (that is, from all possible
loop-switchings which could be applied to $P$). This switching is
rejected with some probability, otherwise it is accepted and performed.
The rejection probability is carefully chosen to ensure that if $P$ has
a uniform distribution over $\mathcal{C}_{\ell,d}$ then the resulting
configuration has a uniform distribution over $\mathcal{C}_{\ell-1,d}$.
If rejection occurs at any step then the entire algorithm restarts from the 
beginning.

When a configuration is reached with no loops, any double pairs are
removed one by one using double-switchings, again with a rejection probability
chosen to maintain uniformity.
Finally, when the current configuration $P$ belongs to $\mathcal{C}_{0,0}$
it is simple, and the algorithm outputs $G(P)$ and terminates.
The algorithm is given in pseudocode in Figure~\ref{fig:McKW}.

\begin{figure}[ht!]
\begin{center}
\hspace{\dimexpr-\fboxrule-\fboxsep\relax}\fbox{
\begin{minipage}{0.8\textwidth}
\begin{tabbing}
  \textsc{McKay--Wormald algorithm}\\
    X \= \kill
  \> XXXs \= \kill 
  \emph{Input:}\>\> graphical sequence $\kvec$\\
  \emph{Output:}\>\> element of $G(\kvec)$\\
  \\
  repeat\\
     X \= \kill
   \> choose $P\in\mathcal{P}(\kvec)$ u.a.r. \\
  until $P$ is good\\[0.5ex]
  \# \emph{remove loops}\\
  while $P$ has at least one loop \\
  \> obtain $P'$ from $P$ by performing a loop-switching chosen u.a.r. \\
  \> calculate the rejection probability $q_{\text{loop}}(P,P')$\\
  \> {\bf restart} with probability $q_{\text{loop}}(P,P')$; otherwise $P:= P'$\\[0.5ex]
  \# \emph{remove double pairs}\\
  while $P$ has at least one double pair\\
   \> obtain $P'$ from $P$ by performing a double-switching chosen u.a.r. \\
  \> calculate the rejection probability $q_{\text{double}}(P,P')$\\
  \> {\bf restart} with probability $q_{\text{double}}(P,P')$; otherwise $P:= P'$\\
  output $G(P)$
\end{tabbing}
\end{minipage}}
\caption{High-level description of the McKay--Wormald algorithm}
\label{fig:McKW}
\end{center}
\end{figure}

To complete the specification of the McKay--Wormald
algorithm, we must define the
rejection probabilities $q_{\text{loop}}$ and $q_{\text{double}}$.
For $P\in\mathcal{P}^\ast(\kvec)$, let $f(P,X)$ denote the number of 
possible $X$-switchings $P\mapsto P'$ which may be applied to $P$,
for $X\in \{\text{loop}, \text{double}\}$.  Similarly, let $b(P',X)$ be
the number of ways to produce $P'$ using an $X$-switching $P\mapsto P'$,
for all $P'\in\mathcal{P}^\ast(\kvec)$ and $X\in \{\text{loop},\text{double}\}$.
McKay and Wormald~\cite[Lemma~4]{MW90} gave expressions
$\underline{m}(\ell,d,\text{loop})$ and $\underline{m}(d,\text{double})$,
omitted here, such that for all $P\in\mathcal{C}_{\ell,d}$,
\begin{equation}
\label{loop-upper}
 f(P,\text{loop}) \leq \overline{m}(\ell,\text{loop})= 2\ell\, M^2,\qquad
  b(P,\text{loop}) \geq \underline{m}(\ell,d,\text{loop}) 
\end{equation}
and for all $P\in\mathcal{C}_{0,d}$,
\begin{equation} 
\label{double-upper}
 f(P,\text{double}) \leq \overline{m}(d,\text{double}) = 4d\, M^2,\qquad
  b(P,\text{double}) \geq \underline{m}(d,\text{double}).
\end{equation}
Regarding the lower bounds $\underline{m}(\cdot)$, 
we will only need the fact that they
are positive when $\kmax^4=O(M)$, for all $\ell \leq B_1$ and
all $d\leq B_2$.

The upper bounds in (\ref{loop-upper}) and (\ref{double-upper})
arise by counting the number of ways to choose a tuple of 
points (6 points for a loop-switching and 8 points for a double-switching)
which satisfy some constraints of the switching
and not others: typically the required pairs must be present,
but we do not check that all cells involved in the switching are 
distinct, or that the switching does not introduce any new loops or
repeated pairs.  For the lower bounds, we require an upper bound on
the number of bad choices of tuples, so that this may be subtracted.
When the degrees get too high, the number of bad choices increases
and there will be a lot of variation in this number, making the
estimates less precise.

The rejection probabilities  are defined by
\begin{equation}
\left.
\begin{aligned}
   q_{\text{loop}}(P,P') &= 1 - \frac{f(P,\text{loop})\, \underline{m}(\ell-1,d,\text{loop})}
                {\overline{m}(\ell,\text{loop})\, b(P',\text{loop})},\\
   q_{\text{double}}(P,P') &= 1 - \frac{f(P,\text{double})\, \underline{m}(d-1,\text{double})} {\overline{m}(d,\text{double})\, b(P',\text{double})}\\
\end{aligned}
\right\}
\label{rejection}
\end{equation}
for all $(P,P')\in \mathcal{C}_{\ell,d}\times\mathcal{C}_{\ell-1,d}$
which differ by a loop-switching, and 
all $(P,P')\in \mathcal{C}_{0,d}\times \mathcal{C}_{0,d-1}$ which differ
by a double-switching, respectively.  
These probabilities are well-defined if the lower bounds $\underline{m}(\cdot)$ are positive. 

\begin{lemma} \emph{\cite[Theorem 2]{MW90}}\,
If $\kmax^4 = O(M)$ then the output of the 
McKay--Wormald algorithm has uniform distribution over $G(\kvec)$.
\label{lem:MW-uniform}
\end{lemma}

\begin{proof}
As mentioned earlier, the condition $\kmax^4=O(M)$ implies
that the lower bounds $\underline{m}(\cdot)$ are positive, and hence
the rejection probabilities are well-defined.
The initial good configuration $P$ is distributed uniformly over 
$\mathcal{P}^\ast(\kvec)$.  Hence, if the initial configuration
belongs to $\mathcal{C}_{\ell,d}$ then it has the uniform distribution over 
$\mathcal{C}_{\ell,d}$.
We prove by induction that if a switching $P\mapsto P'$
is accepted and $P$ has the uniform distribution over some set $\mathcal{C}_{\ell,d}$,
then $P'$ has the uniform distribution over the codomain of that switching.
(The codomain is $\mathcal{C}_{\ell-1,d}$ if the switching is a loop-switching, 
while for a double-switching $\ell=0$ and the codomain is $\mathcal{C}_{0,d-1}$.) 
For ease of notation we prove this for double-switchings, and note
that the same argument holds for loops-switchings.
For all $P'\in\mathcal{C}_{0,d-1}$, the probability that the proposed switching
is not rejected and results in $P'$ is given by
\[
\Pr(P') = \sum_{\substack{P\in\mathcal{C}_{0,d}\\P\mapsto P'}} 
  \frac{\Pr(P)}{f(P,\text{double})}\, \big(1-q_{\text{double}}(P,P')\big).\]
The sum is over all configurations $P\in\mathcal{C}_{0,d}$ such that $P'$
can be obtained from $P$ using a double-switching, and the factor 
$1/f(P,\text{double})$ is the probability that this particular 
double-switching is chosen to be applied to $P$. 
Substituting the value of the rejection probability from (\ref{rejection}),
and using the assumption
that $P$ is uniformly distributed over $\mathcal{C}_{0,d}$, we find that 
\[ \Pr(P') = \frac{\underline{m}(d-1,\text{double})}{|\mathcal{C}_{0,d}|\, \overline{m}(d,\text{double})}\,
    \sum_{\substack{P\in\mathcal{C}_{0,d}\\P\mapsto P'}} \frac{1}{b(P',\text{double})}.\]
But the number of summands is precisely $b(P',\text{double})$, so the sum evaluates to~1 and
we conclude that
\[ \Pr(P') = \frac{\underline{m}(d-1,\text{double})}{|\mathcal{C}_{0,d}|\, \overline{m}(d,\text{double})}.
\]
This depends only on $d$, and not on the particular configuration 
$P'\in\mathcal{C}_{0,d-1}$.  Hence every element of $\mathcal{C}_{0,d-1}$ is
equally likely to be produced after the double-switching, proving that
the uniform distribution is maintained after each accepted switching step.
Thus, by induction, at the end of the algorithm $P$ is a uniformly random 
element of $\mathcal{C}_{0,0}$. It follows that $G(P)$ is a uniformly random
element of $G(\kvec)$, as claimed.
\end{proof}

The previous result shows that the output of the McKay--Wormald algorithm 
is always correct.  But what conditions on $\kvec$ are needed
for the algorithm to be efficient?
If the degrees become too large then it becomes unlikely that
the randomly-chosen initial configuration is good, and there will be too much
variation in the parameters $f(P,X)$, $b(P',X)$, leading to large rejection 
probabilities.

\begin{theorem} \emph{\cite[Theorem 3]{MW90}}\,
Suppose that $\kvec$ is a graphical degree sequence with $\kmax = O(M^{1/4})$.
The McKay--Wormald algorithm for sampling from $G(\kvec)$
can be implemented so that it
has expected runtime $O(\kmax^2 M^2) = O(\kmax^4 n^2)$.
If $\kvec=(k,k\ldots, k)$ is regular then there is an implementation
with expected runtime $O(k^3 n)$, under the assumption that $k = O(n^{1/3})$.
\label{thm:MW-runtime}
\end{theorem}

\begin{proof} (Sketch.)\
Recall that a randomly chosen element of $\mathcal{P}(\kvec)$ is
good with probability at $c$ when $\kmax^4=O(M)$, for some
constant $c\in (0,1)$.   
Hence it takes 
expected time $O(M)$ to produce a uniformly-random element of
$\mathcal{P}^\ast(\kvec)$. 
McKay and Wormald prove that the probability that there is no restart during
the loop-switchings and doubles-switchings is $1-o(1)$ when $\kmax^4=o(M)$,
and is bounded below by a constant when $\kmax^4=\Theta(M)$.  

It remains to
consider the cost of performing the switching operations.
Suppose that at some point in the execution of the algorithm,
the current configuration is $P$.
To choose a potential switching of the appropriate type, 
we can select the points of a randomly chosen 
loop or double pair, in a random order, and then choose the points of two
other pairs, in a random order.   The number of ways to make this selection is
\emph{exactly} given by the relevant upper bound 
from (\ref{loop-upper}) or (\ref{double-upper}),
and the probability that the result $P'$ of this switching is a valid
configuration in the codomain (that is, only the chosen loop/double pair
has been removed, and no additional defects have been introduced) is
exactly $f(P,X)/\overline{m}(x,X)$, where $(x,X) \in\{(\ell,\text{loop}),\,
(d,\text{double})\}$.  This means that the value of $f(P,X)$
does not need to be calculated.  

However, we do need to calculate the value
$b(P',X)$ precisely for the proposed switching $P\mapsto P'$, 
in order to restart with the correct probability.
This can be done by maintaining some information about numbers of 
small structures
in the configuration, which is initialised before any switchings
have been performed, and updated after each switching operation. 
The initialisation
takes runtime $O(\kmax^2 M_2^2)$, which dominates the expected
time required for the updates from each switching.
See~\cite[Theorem 3]{MW90} for more details.  

In the $k$-regular case a further efficiency is possible, 
but with a much more complicated implementation,
as explained in the proof of~\cite[Theorem 4]{MW90}.
\end{proof}

McKay and Wormald also explained how to modify their algorithm
to sample bipartite graphs with given degrees, uniformly
at random, see~\cite[Section 6]{MW90}.  The expected runtime
of the uniform sampler for bipartite graphs with given degrees
is $O(\kmax^4 n^2)$ when $\kmax = O(M^{1/4})$, where (as usual)
$\kmax$ denotes the maximum degree.

\subsection{Improvements and extensions}\label{s:switching-recent}

Starting from the McKay--Wormald algorithm, Gao and Wormald~\cite{GW} 
introduced several new ideas which culminated in an 
algorithm for uniformly sampling $k$-regular graphs, which they
called \textsc{REG}.  The expected runtime of the Gao--Wormald algorithm
is $O(k^3 n)$ when $k=o(\sqrt{n})$.  This is a significant increase
in the allowable range of $k$ compared with the McKay--Wormald algorithm.

In order to handle degrees beyond $O(n^{1/3})$, Gao and Wormald must
deal with triple pairs, as well as loops and double pairs.  So the set
of good configurations is redefined to allow ``not too many'' triple pairs
(but still ruling out any pairs of multiplicity four or higher)
and a new switching phase is performed to remove triple pairs one by one.
However, it turns out that triple pairs are easily handled.  
In fact, the first two phases of the algorithms (removing loops and
removing triples, respectively) proceed as in the McKay--Wormald
algorithm.  As the double pairs are the most numerous ``defect'',
new ideas are required in phase 3, where double pairs must be removed.

The innovations introduced by Gao and Wormald in~\cite{GW} are
designed to reduce the probability of a rejection during a 
double-switching step.
These ideas are described in a very general setting in~\cite{GW}, for
ease of applications to other problems.  Here, we given an overview of these
ideas in the context of the double-switching, performed on configurations
which have no loops, no triple pairs and  at most $B_2$ double pairs, where
$B_2=\lfloor (1+\gamma)(k-1)^2/4\rfloor$
for some sufficiently small constant $\gamma>0$.

Since we now discuss only double-switchings,
we drop ``double'' from our notation.
Write $\mathcal{C}_d$ for the set of good configurations with no
loops, no triple pairs and exactly $d$ double pairs.
Observe from (\ref{rejection}) that the probability that a proposed
switching $P\mapsto P'$
is \emph{not} rejected is a product of a factor $f(P)/\overline{m}(d)$
which depends only on $P\in\mathcal{C}_{d}$, 
and a factor $\underline{m}(d-1)/b(P')$
which depends only on $P'\in\mathcal{C}_{d-1}$.
Gao and Wormald aim to reduce both the probability of \emph{forward rejection},
or \emph{f-rejection} (which depends only on $P$), and the probability
of \emph{backwards rejection}, or \emph{b-rejection}, which depends only
on $P'$.

To reduce the likelihood of an f-rejection, Gao and Wormald allow some 
double-switchings which would be rejected in the McKay--Wormald algorithm.
The aim is to bring the values of $f(P)$ closer to the upper bound
$\overline{m}(d)$.
Specifically, a double-switching which introduces exactly one new
double pair is allowed, as illustrated in Figure~\ref{fig:double-classB}.  
The original double-switching, shown on the right of Figure~\ref{fig:switchings}, is known as a type I, class A switching, while switching in
Figure~\ref{fig:double-classB} is a type I, class B switching.

\begin{figure}[ht!]
\begin{center}
\begin{tikzpicture}[scale=0.7]
\draw (0,0) circle (0.5); \draw (4,0) circle (0.5);
\draw (2,0) circle (0.8 and 0.5); \draw (2,2.1) circle (0.8 and 0.7);
\draw (0,2.2) circle (0.5 and 0.7); \draw (4,2) circle (0.5);
\draw [fill] (0,0) circle (0.1); \draw [fill] (4,0) circle (0.1);
\draw [fill] (2,2.4) circle (0.1); \draw [fill] (0,2.4) circle (0.1);
\draw [fill] (1.75,0) circle (0.1); \draw [fill] (1.75,2) circle (0.1);
\draw [fill] (2.25,0) circle (0.1); \draw [fill] (2.25,2) circle (0.1);
\draw [fill] (0,2) circle (0.1); \draw [fill] (4,2) circle (0.1);
\draw [-] (0,0) -- (0,2)  (1.75,0) -- (1.75,2) (2.25,0) -- (2.25,2)  (4,0) -- (4,2) (0,2.4) -- (2,2.4);
\draw [->, line width=3pt] (5.5,1) -- (6.5,1);
\begin{scope}[shift={(8,0)}]
\draw (0,0) circle (0.5); \draw (4,0) circle (0.5);
\draw (2,0) circle (0.8 and 0.5); \draw (2,2.1) circle (0.8 and 0.7);
\draw [fill] (2,2.4) circle (0.1); \draw [fill] (0,2.4) circle (0.1);
\draw (0,2.2) circle (0.5 and 0.7); \draw (4,2) circle (0.5);
\draw [fill] (0,0) circle (0.1); \draw [fill] (4,0) circle (0.1);
\draw [fill] (1.75,0) circle (0.1); \draw [fill] (1.75,2) circle (0.1);
\draw [fill] (2.25,0) circle (0.1); \draw [fill] (2.25,2) circle (0.1);
\draw [fill] (0,2) circle (0.1); \draw [fill] (4,2) circle (0.1);
\draw [-] (0,0) -- (1.75,0)  (2.25,0) -- (4,0) (2.25,2) -- (4,2)  (0,2) -- (1.75,2) (0,2.4) -- (2,2.4);
\end{scope}
\end{tikzpicture}
\caption{A type I, class B double-switching $\mathcal{C}_d\to\mathcal{C}_d$}
\label{fig:double-classB}
\end{center}
\end{figure}
Next, Gao and Wormald observed that some configurations in 
$\mathcal{C}_{d-1}$ are less likely to be produced by a
(type I) double-switching than others.  These configurations
bring down the lower bound $\underline{m}(d-1)$ and hence
increase the b-rejection probability for every element of $\mathcal{C}_{d-1}$.
For this reason, Gao and Wormald introduced another new switching, called a 
\emph{type II} switching, which actually increases the number of double pairs
by one, as shown in Figure~\ref{fig:double-typeII}.  All type II switchings
have class B.

\begin{figure}[ht!]
\begin{center}
\begin{tikzpicture}[scale=0.7]
\draw (0,0) circle (0.5); \draw (4,0) circle (0.5);
\draw (2,0) circle (0.8 and 0.5); \draw (2,2.2) circle (0.7 and 0.7);
\draw (0,2.2) circle (0.5 and 0.7); \draw (4,2.2) circle (0.5 and 0.7);
\draw [fill] (0,0) circle (0.1); \draw [fill] (4,0) circle (0.1);
\draw [fill] (1.75,2.4) circle (0.1); \draw [fill] (2.25,2.4) circle (0.1);
\draw [fill] (0,2.4) circle (0.1); \draw [fill] (4,2.4) circle (0.1); 
\draw [fill] (1.75,0) circle (0.1); \draw [fill] (1.75,2) circle (0.1);
\draw [fill] (2.25,0) circle (0.1); \draw [fill] (2.25,2) circle (0.1);
\draw [fill] (0,2) circle (0.1); \draw [fill] (4,2) circle (0.1);
\draw [-] (0,0) -- (0,2)  (1.75,0) -- (1.75,2) (2.25,0) -- (2.25,2)  (4,0) -- (4,2) (0,2.4) -- (1.75,2.4) (2.25,2.4) -- (4,2.4);
\draw [->, line width=3pt] (5.5,1) -- (6.5,1);
\begin{scope}[shift={(8,0)}]
\draw (0,0) circle (0.5); \draw (4,0) circle (0.5);
\draw (2,0) circle (0.8 and 0.5); \draw (2,2.2) circle (0.7 and 0.7);
\draw [fill] (1.75,2.4) circle (0.1); \draw [fill] (2.25,2.4) circle (0.1); 
\draw [fill] (0,2.4) circle (0.1); \draw [fill] (4,2.4) circle (0.1); 
\draw (0,2.2) circle (0.5 and 0.7); \draw (4,2.2) circle (0.5 and 0.7);
\draw [fill] (0,0) circle (0.1); \draw [fill] (4,0) circle (0.1);
\draw [fill] (1.75,0) circle (0.1); \draw [fill] (1.75,2) circle (0.1);
\draw [fill] (2.25,0) circle (0.1); \draw [fill] (2.25,2) circle (0.1);
\draw [fill] (0,2) circle (0.1); \draw [fill] (4,2) circle (0.1);
\draw [-] (0,0) -- (1.75,0)  (2.25,0) -- (4,0) (2.25,2) -- (4,2)  (0,2) -- (1.75,2) (0,2.4) -- (1.75,2.4) (2.25,2.4) -- (4,2.4);
\end{scope}
\end{tikzpicture}
\caption{A type II, class B double-switching $\mathcal{C}_d\to\mathcal{C}_{d+1}$}
\label{fig:double-typeII}
\end{center}
\end{figure}

To perform a switching step, from current configuration $P\in\mathcal{C}_{d}$,
first the type $\tau\in\{ I,II\}$ of switching is chosen, according
to a probability distribution $\rho$ (with a small restart probability
if no type is chosen).  Next, a type $\tau$ switching 
$P\mapsto P'$ is proposed, chosen randomly from all $f_{\tau}(P)$ type $\tau$
switchings available in $P$.  The f-rejection probability is
$1-f_\tau(P)/\overline{m}_\tau(d)$, where $\overline{m}_\tau(d)$ is an 
upper bound on $f_\tau(P)$ over all $P\in\mathcal{C}_{d}$.
Let $d'\in \{0,\ldots, B_2\}$ be the unique index
such that $P'\in\mathcal{C}_{d'}$.
The class $\alpha\in\{A,B\}$ of the proposed switching $P\mapsto P'$
can now be observed, and the b-rejection probability is 
$1-b_\alpha(P')/\underline{m}_\alpha(d')$, where 
$\underline{m}_\alpha(d')$ is a lower bound on $b_\alpha(P')$
over all $P'\in\mathcal{C}_{d'}$.
If there is no f-rejection or b-rejection then the proposed
switching is accepted and $P'$ becomes the current configuration.
As soon as an element $P\in\mathcal{C}_0$ is reached, the algorithm
stops with output $G(P)$.
Here we see that the f-rejection probability depends on $P$ and the chosen
type, while the b-rejection probability depends on the outcome $P'$ and
the class $\alpha$ of the proposed switching from $P$.

Rather than maintaining a uniform distribution after each switching,
as in the McKay--Wormald algorithm, the goal in the Gao--Wormald algorithm
is to ensure that the expected number of visits to each 
configuration $P\in\mathcal{C}_d$, over the course of (the doubles-reducing
phase of) the algorithm, 
depends only on $d$ and is independent of $P$.  In particular,
this guarantees that each element of $\mathcal{C}_0$ is equally likely,
and hence the output of the algorithm is a uniformly random element of $G(n,k)$.

In~\cite[Lemma 6 and Lemma 8]{GW}, Gao and Wormald gave expressions for
$\overline{m}_\tau(d)$, $\underline{m}_\alpha(d)$ and 
showed how to choose values for $\rho_\tau(d)$
satisfying a certain system of linear equations.  By~\cite[Lemma~5]{GW},
 when the parameters $\rho_{\tau}(d)$ satisfy
this system of equations then
the last element visited by the algorithm is distributed uniformly at
random from $\mathcal{C}_0$, assuming that no rejection occurs.
Furthermore, the solution can be chosen to satisfy 
$\rho_I(d) = 1-\varepsilon>0$ for all $1\leq d\leq B_2$,
where $\varepsilon = O(k^2/n^2)$.
Since $k=o(n^{1/2})$ this means that almost every step is a
``standard'' double-switching (type I, class A) switching.

Having set these parameter values, the algorithm is completely specified.
It remains to show that the probability of rejection during the course
of the algorithm is $o(1)$, which requires careful analysis.
The runtime analysis is very similar to Theorem~\ref{thm:MW-runtime},
resulting in the following.

\begin{theorem}
If $1\leq k=o(\sqrt{n})$ then the Gao--Wormald algorithm \textsc{REG}
is a uniform sampler from $G(n,k)$, and can be implemented
 with expected runtime $O(k^3 n)$.
\label{thm:GW}
\end{theorem}

Recent work of Armand, Gao and Wormald~\cite{AGW} which gives an even
more efficient uniform sampler for the same range of $k$ is discussed 
in Section~\ref{s:incremental}.

\bigskip

A $k$-\emph{factor} is a $k$-regular spanning subgraph of a given
graph.  Gao and Greenhill~\cite{GG-dFactor} used the Gao--Wormald
framework to give algorithms for sampling $k$-factors of a given
graph $H_n$ with $n$ vertices, under various conditions on $k$
and the maximum degree $\Delta$ of the complement $\overline{H}_n$ of $H_n$.
The edges of the complement of $H_n$ can be thought of as
``forbidden edges'', and we want to sample $k$-regular graphs
with no forbidden edges.  

\begin{theorem} \emph{\cite[Theorem~1.1 and 1.2]{GG-dFactor}}\
\label{thm:GG-dFactor}
Let $H_n$ be a graph on $n$ vertices such that $\overline{H}_n$ has
maximum degree $\Delta$. 
\begin{itemize}
\item There is an algorithm which  produces a uniformly random 
$k$-factor of $H_n$,  and has expected runtime $O((k+\Delta)^3 n)$
if $(k+\Delta)k\Delta = o(n)$. 
\item
Now suppose that $H_n$ is $(n-\Delta-1)$-regular.
There is an algorithm which generates a uniformly random $k$-factor of $H_n$ and 
has expected runtime
\[ O\big( (k+\Delta)^4 (n + \Delta)^3 + (k + \Delta)^8\, k^2\Delta^2/n
     + (k + \Delta)^{10}\, k^2 \Delta^3/n^2\big)\]
if $k^2 + \Delta^2 = o(n)$. 
\end{itemize}
\end{theorem}

In~\cite{GG-dFactor}, the algorithms described in Theorem~\ref{thm:GG-dFactor} are  called
\textsc{FactorEasy} and \textsc{FactorUniform}, respectively. 
Previously the only algorithm for this problem was a rejection
algorithm of Gao~\cite{Gao-rejection} which has expected linear
runtime when $k=O(1)$ and $\overline{H}_n$ has at most a linear number
of edges (but the maximum degree of $\overline{H}_n$ can be linear).

\subsubsection{Asymptotically-uniform algorithms based on switchings}

In~\cite[Theorem 3]{GW}, Gao and Wormald described an algorithm 
\textsc{REG}$^*$ which performs
asymptotically-uniform sampling from $G(n,k)$ in expected runtime $O(kn)$.
This algorithm is obtained from \textsc{REG}
by never performing any rejection steps and never performing any class~B
switchings.  (So only loop-switchings, triple-switchings and
type I, class A double-switchings will be used.) 
This is more efficient as computation of the
b-rejection probabilities is the most costly part of the algorithm.
The output of \textsc{REG}$^*$ is within total variation distance
$o(1)$ of uniform when $k=o(\sqrt{n})$.
Indeed, Gao and Wormald remark that the McKay--Wormald algorithm
can be modified in the same way, giving an asymptotically-uniform
sampling algorithm with expected runtime $O(M)$ whenever $\kmax = O(M^{1/4})$.

A similar performance was obtained by Zhao~\cite{zhao} using a slightly
different approach, involving the use of a Markov chain to make
local modifications starting from $G(P)$, where $P$ is a uniformly
random element of $\mathcal{P}(\kvec)$.

Recently, Janson~\cite{Janson2019} introduced and analysed
the following switching-based algorithm for asymptotically-uniform sampling
from $G(\kvec)$.  Say that a pair in a configuration is \emph{bad} if it
is a loop or part of a double pair.  (Recall that, as we have defined
it here, ``double pair'' does not
mean that the multiplicity of the corresponding edge is exactly two: only
that the multiplicity is at least two.)

\begin{figure}[ht!]
\begin{center}
\hspace{\dimexpr-\fboxrule-\fboxsep\relax}\fbox{
\begin{minipage}{0.8\textwidth}
\begin{tabbing}
  \textsc{Janson algorithm}\\
    X \= \kill
  \> XXXs \= \kill 
  \emph{Input:}\>\> graphical sequence $\kvec$\\
  \emph{Output:}\>\> element of $G(\kvec)$, denoted $\widehat{G}$\\
  \\
   choose $P\in\mathcal{P}(\kvec)$ u.a.r. \\
repeat\\
     X \= \kill
  \> choose and orient a bad pair $pp'$ in $P$ u.a.r. \\
  \> choose and orient a distinct pair $qq'$, u.a.r.\\
  \> delete pairs $pp', qq'$ from $P$, and replace with pairs
   $pq, p'q'$\\
  until $G(P)$ is simple\\
  output $\widehat{G} = G(P)$
\end{tabbing}
\end{minipage}}
\caption{Janson algorithm, corresponding to the switched configuration model}
\label{fig:janson}
\end{center}
\end{figure}
Starting from a uniformly random configuration
$P\in\mathcal{P}(\kvec)$, if $G(P)$ is simple then we output $G(P)$.
Otherwise, choose a pair uniformly at random from the set of all bad pairs
in $P$.  Next, choose a pair uniformly at random
from the set of all other pairs in $P$. 
Update $P$ by removing these two pairs and replacing them by two other 
pairs using the same four points, chosen uniformly at random.
(In the pseudocode, this is done by randomly ordering the points in each
chosen pair.)
This gives a new configuration in $\mathcal{P}(\kvec)$. 
At each step, the switching removes the chosen bad pair, 
and may cause other pairs to stop being bad or to become bad.  
Repeatedly apply the switching step until $G(P)$ is simple, and let $\widehat{G}=G(P)$ denote the output graph.   
This algorithm is shown in Figure~\ref{fig:janson}.
Janson calls the resulting probability
space the \emph{switched configuration model}. This is a non-uniform
probability space over $G(\kvec)$.

The switchings used in this process are illustrated in 
Figure~\ref{fig:switchings-Janson}.  They were used by
McKay~\cite{McKay85} in a very early application of the switching method,
and are simpler than the switchings in Figure~\ref{fig:switchings}. 
Note however that Janson does not insist that the cells involved in the
switching are all distinct, or that the new pairs $pq$, $p'q'$ do not
increase the multiplicity of an edge.  So 
Figure~\ref{fig:switchings-Janson} should be interpreted differently
to Figure~\ref{fig:switchings}: in Figure~\ref{fig:switchings}, it is implied 
that all illustrated
cells are distinct and that all new pairs lead to edges with multiplicity~1.
\begin{figure}[ht!]
\begin{center}
\begin{tikzpicture}[scale=0.6]
\draw (-0.2,2) circle (0.8 and 0.5); \draw (2.2,2) circle (0.8 and 0.5);
\draw (1,3.5) circle (1.2 and 0.6);
\draw [fill] (0,2) circle (0.1); \draw [fill] (2,2) circle (0.1);
\draw [fill] (0.75,3.5) circle (0.1); \draw [fill] (1.25,3.5) circle (0.1);
\draw [-] (0,2) -- (2,2);
\draw [-] plot [smooth] coordinates {(0.75,3.5) (0.5,4.3) (1.5,4.3) (1.25,3.5)};
\node [left] at (0.7,3.4) {$p$}; \node [right] at (1.3,3.5) {$p'$};
\node [left] at (-0.1,1.9) {$q$}; \node [right] at (2.05,2) {$q'$};
\draw [->, line width=3pt] (3.5,3) -- (4.5,3);
\begin{scope}[shift={(6,0)}]
\draw (-0.2,2) circle (0.8 and 0.5); \draw (2.2,2) circle (0.8 and 0.5);
\draw (1,3.5) circle (1.2 and 0.6);
\draw [fill] (0,2) circle (0.1); \draw [fill] (2,2) circle (0.1);
\draw [fill] (0.75,3.5) circle (0.1); \draw [fill] (1.25,3.5) circle (0.1);
\draw [-] (0,2) -- (0.75,3.5)  (1.25,3.5)-- (2,2);
\node [left] at (0.7,3.4) {$p$}; \node [right] at (1.3,3.5) {$p'$};
\node [left] at (-0.1,1.9) {$q$}; \node [right] at (2.05,2) {$q'$};
\end{scope}
\begin{scope}[shift={(12,2)}]
\draw (-0.2,0) circle (0.7 and 0.5); \draw (2.2,0) circle (0.7 and 0.5); 
\draw (-0.2,2) circle (0.7 and 0.8); \draw (2.2,2) circle (0.7 and 0.8);
\draw [fill] (0,0) circle (0.1); \draw [fill] (2,0) circle (0.1); 
\draw [fill] (0,1.75) circle (0.1); \draw [fill] (2,1.75) circle (0.1);
\draw [fill] (0,2.25) circle (0.1); \draw [fill] (2,2.25) circle (0.1);
\draw [-] (0,0) -- (2,0)  (0,1.75) -- (2,1.75) (0,2.25) -- (2,2.25);
\draw [->, line width=3pt] (3.5,1) -- (4.5,1);
\node [left] at (0.0,-0.1) {$q$}; \node [right] at (2.0,-0.1) {$q'$};
\node [left] at (0.0,1.75) {$p$}; \node [right] at (2.0,1.75) {$p'$};
\begin{scope}[shift={(6,0)}]
\draw (-0.2,0) circle (0.7 and 0.5); \draw (2.2,0) circle (0.7 and 0.5);
\draw (-0.2,2) circle (0.7 and 0.8); \draw (2.2,2) circle (0.7 and 0.8);
\draw [fill] (0,0) circle (0.1); \draw [fill] (2,0) circle (0.1); 
\draw [fill] (0,1.75) circle (0.1); \draw [fill] (2,1.75) circle (0.1);
\draw [fill] (0,2.25) circle (0.1); \draw [fill] (2,2.25) circle (0.1);
\draw [-] (0,0) -- (0,1.75)   (2,0) -- (2,1.75) (0, 2.25) -- (2,2.25);
\node [left] at (0.0,-0.1) {$q$}; \node [right] at (2.0,-0.1) {$q'$};
\node [left] at (0.0,1.75) {$p$}; \node [right] at (2.0,1.75) {$p'$};
\end{scope}
\end{scope}
\end{tikzpicture}
\caption{Possible switchings in Janson's algorithm, with chosen points labelled}
\label{fig:switchings-Janson}
\end{center}
\end{figure}

Janson proved the following result~\cite[Theorem 2.1]{Janson2019}.
Recall the definition of $R=R(\kvec)$ from Theorem~\ref{thm:config}.

\begin{theorem} \emph{\cite{Janson2019}}
\label{thm:janson-switching}
Suppose that $\kvec$ is a graphical degree sequence which satisfies
\begin{equation}
\label{janson-conditions}
 \kmax = o(n^{1/2}),\qquad M = \Theta(n), \qquad R = O(n)
\end{equation}
and let $\widehat{G}$ denote the output of the switched
configuration model for the degree sequence $\kvec$.
Then the distribution of $\widehat{G}$ is within total variation
distance $o(1)$ of uniform. With high probability, the runtime
is $O(M)$ as only $O(1)$ switching steps are required. 
\end{theorem}

Janson remarks that the bad pair may also be chosen deterministically
according to some rule, such as lexicographically.  This would lead
to a slightly different distribution on the output graph,
but the conclusion of Theorem~\ref{thm:janson-switching} would still hold.

Furthermore, Janson~\cite[Corollary 2.2 and Corollary 2.3]{Janson2019} 
proved that under the same conditions (\ref{janson-conditions}),
statements about convergence in probability and convergence
in distribution which are true for $\widehat{G}$ are also true
for uniformly-random elements of $G(\kvec)$.  

\subsubsection{Graphs with power-law degree distributions}

Heavy-tailed distributions are often observed in real-world 
networks~\cite{CSN,newman2003},
but are difficult to sample as they are far from regular and their
maximum degree is too high for the sampling algorithms we
have seen so far.  In~\cite{GW-power},
Gao and Wormald showed how to adapt their approach to 
degree sequences which satisfy the following definition.

\begin{Def}
\cite[Definition~1]{GW-power}\
The degree sequence $\kvec$ is \emph{power-law distribution-bounded}
with parameter $\gamma >1$ if the minimum entry in $\kvec$ is at
least~1, and there is a constant $C>0$ independent of $n$ such that
the number of entries of $K$ which are at least $i$ is at most 
$C n i^{1-\gamma}$ for all $i\geq 1$.   
\label{def:plib}
\end{Def}

Other definitions of power-law degree sequences can be found in
the literature, but some only allow maximum degree $O(n^{1/\gamma})$,
which is $o(n^{1/2})$ when $\gamma\in (2,3)$.
Definition~\ref{def:plib} is more realistic as it allows higher degrees,
as observed in real-world networks.
Gao and Wormald~\cite[equation (4)]{GW-power} noted that if $\kvec$
is power-law distribution-bounded with parameter $\gamma$ then
\begin{equation}
\label{power-law-facts}
 \kmax = O(n^{1/(\gamma-1)}), \qquad M = \Theta(n), \qquad M_2 = O(n^{2/(\gamma-1)}),
\end{equation}
where $M_2 = M_2(\kvec)$ is given in Definition~\ref{def:good-set}.

\bigskip

The most relevant range of $\gamma$ for real-world networks is
$\gamma\in (2,3)$.  As observed
by Gao and Wormald~\cite{GW-power}, when $\gamma>3$ it is easy
to sample uniformly from $G(\kvec)$.  We provide a brief proof here.

\begin{lemma} \emph{\cite{GW-power}}\
\label{lem:power-high}
Suppose that $\kvec$ is a power-law distribution-bounded degree
sequence with $\gamma >3$.  Then the configuration model (Figure~\ref{fig:config})
gives a polynomial-time uniform sampler for $G(\kvec)$ with
expected runtime $O(M)$.
\end{lemma}

\begin{proof}
It follows from (\ref{power-law-facts}) that
$\kmax^2 = o(M)$ and 
$R = \sum_{j\in [n]} k_j^2 = M_2 + M = \Theta(M)$.
The proof is completed by applying Theorem~\ref{thm:config}. 
\end{proof}

While uniform sampling is easy when $\gamma> 3$, it is a very
challenging problem when $\gamma\in (2,3)$.
To cope with the very high maximum degree when $\gamma < 3$,
Gao and Wormald utilised 6 different types of switching (all of the same class).
First, they focussed on removing
``heavy'' edges or loops, where  an edge is \emph{heavy} if both its endvertices
have high degree.  They also introduced a new kind
of rejection, called \emph{pre-b-rejection}, which is used to equalise
the number of ways to choose some additional pairs which are needed
to perform some of the switchings.  
They described a uniform sampler \textsc{PLD}, and an asymptotically-uniform
sampling algorithm called \textsc{PLD}$^*$, obtaining the following
result~\cite[Theorem 2, Theorem 3]{GW-power} when the parameter $\gamma$
is a little less than~3.
These are the first rigorously-analysed algorithms which can efficiently 
sample graphs with 
a realistic power-law degree distribution for some values of $\gamma$ below~3.

\begin{theorem} \emph{\cite{GW-power}}\
\label{thm:GW-power}
Suppose that $\kvec$ is a power-law distribution-bounded degree
sequence with parameter $\gamma$ such that
\[ \gamma > \dfrac{21 + \sqrt{61}}{10} \approx 2.881.\]
The algorithm \textsc{PLD} is a uniform sampler for $G(\kvec)$
with expected runtime $O(n^{4.081})$.
The algorithm \textsc{PLD}$^*$ performs asymptotically-uniform
sampling from $G(\kvec)$ with  expected runtime $O(n)$.
\end{theorem}

In their analysis, Gao and Wormald used a new parameter, $J(\kvec)$,
which they introduced in the context of asymptotic enumeration
in~\cite{GW-power-enum}.  We have seen that a switching argument
breaks down when the number of bad choices for a given switching
operation becomes too large.  This often involves counting paths
of length two from a given vertex.  In previous work, the
bound $k_{\max}^2$ was often used for this quantity.
(Of course $k_{\max}(k_{\max}-1)$ is more precise but gives
the same asymptotics.)
Instead, Gao and Wormald use the upper bound $J(\kvec)$,
defined as follows.  First, let $\sigma$ be a permutation of
$[n]$ such that $k_{\sigma(1)}\geq k_{\sigma(2)}\geq \cdots \geq k_{\sigma(n)}$.
Then, define
\begin{equation}
\label{def:Jd}
 J(\kvec) = \sum_{j=1}^{k_{\max}} k_{\sigma(j)},
\end{equation}
noting that $k_{\max} = k_{\sigma(1)}$.  
So $J(\kvec)$ is the sum of the $k_{\max}$ largest entries of
$\kvec$, and hence forms an upper bound on the number of 2-paths
from an arbitrary vertex.

If $\kvec$ is regular then $J(\kvec)=\kmax^2$, but when $\kvec$
is far from regular, $J(\kvec)$ can be significantly smaller
than $k_{\max}^2$.    In particular, if $\kvec$ is a power-law
distribution-bounded degree sequence with parameter $\gamma$
then $k_{\max}^2 = n^{2/(\gamma-1)}$, while 
\[ J(\kvec) = O\big(n^{(2\gamma-3)/(\gamma-1)^2}\big) = o(n^{2/(\gamma-1)}).\] 
This bound on $J(\kvec)$ is proved in~\cite[Lemma~5]{GvSS}, and more briefly
in~\cite[equation~(54)]{GW-power}.

The parameter $J(\kvec)$ has  proved very powerful when working
with heavy-tailed degree distributions.  As well as its use in
asymptotic enumeration~\cite{GW-power-enum} and uniform sampling~\cite{GW-power},
it has also been used in the analysis~\cite{GvSS} of the number of 
triangles and the clustering coefficient in a uniformly random element
of $G(\kvec)$, for heavy-tailed degree sequences $\kvec$.
We will encounter $J(\kvec)$ again in Section~\ref{s:stability}.

\subsubsection{Incremental relaxation}\label{s:incremental}

Very recently, Arman, Gao and Wormald~\cite{AGW} introduced
a new approach, called \emph{incremental relaxation},
which allows a more efficient implementation of the b-rejection step. 
In incremental relaxation, the b-rejection is performed iteratively
over several steps, each with its own sub-rejection probability,
such that the sub-rejection probabilities are much easier to calculate
than the overall probability of b-rejection.  Using this idea,
Arman et al.\ obtain improvements on the runtime of the algorithms
in~\cite{GW,GW-power,MW90}, and give an algorithm for uniformly sampling 
bipartite graphs with given degrees when the maximum degree is $O(M^{1/4})$.
We collect their results together below.

\begin{theorem} \emph{\cite[Theorems 1--4]{AGW}}\
Let $\kvec$ be a graphical degree sequence.  There are
 algorithms, called \textsc{INC-GEN}, \textsc{INC-REG},
and \textsc{INC-POWERLAW}, respectively,
which 
perform uniform sampling from $G(\kvec)$
under the following assumptions on $\kvec$, with the stated expected
runtime:
\begin{itemize}
\item  If $\kmax^4 = O(M)$ then the expected runtime of the algorithm
\textsc{INC-GEN} is $O(M)$.
\item If $\kmax=(k,k,\ldots, k)$ is regular and $k = o(n^{1/2})$
then the expected runtime of the algorithm \textsc{INC-REG} is
$O(kn + k^4)$.
\item If $\kvec$ is a power-law distribution-bounded degree sequence
with parameter $\gamma > \frac{21 + \sqrt{61}}{10}\approx 2.881$
then the algorithm \textsc{INC-POWERLAW}  has expected runtime
$O(n)$.
\end{itemize}
Now let $\kvec=(\svec,\tvec)$ be a bipartite degree sequence
with $\kmax = \max\{ \max s_j,\, \max t_i\}$.
If $\kmax$ satisfies $\kmax^4 = O(M)$,
then there is an algorithm, called \textsc{INC-BIPARTITE}, 
which has expected runtime
$O(M)$ and produces a uniformly-random bipartite graph with
degree sequence $\svec$ on one side of the bipartition and $\tvec$
on the other.
\end{theorem}

Observe that incremental relaxation leads to greatly improved
runtimes, for example from $O(\kmax^2 M^2)$ to $O(M)$ for uniform
sampling from $G(\kvec)$ when $\kmax = O(M^{1/4})$, and
from $O(n^{4.081})$ to $O(n)$ for power-law distribution-bounded
degree sequences with $\gamma > 2.882$.

At the time of writing, C code for \textsc{INC-GEN} and \textsc{INC-REG} is available from
Wormald's website\footnote{\texttt{https://users.monash.edu.au/$\sim$nwormald}}.

\section{Markov chain algorithms}\label{s:markov}

In this section we review sampling
algorithms which use the \emph{Markov chain Monte Carlo}  (MCMC) approach.
Here an ergodic Markov chain is defined with the desired stationary
distribution: in our setting, the stationary distribution should be 
\emph{uniform} over $G(\kvec)$, or perhaps over a superset of $G(\kvec)$.
We refer to such algorithms as ``Markov chain algorithms''.

Rather than work steadily towards a particular goal, such as the sequential
algorithms described in Section~\ref{s:SW} or the McKay--Wormald algorithm 
discussed in Section~\ref{s:switchings}, the Markov chains we consider
in this section perform a random walk on $G(\kvec)$, usually by making small
random perturbations at each step.  For example, the \emph{switch chain} chooses
two random edges, deletes them and replaces them with two other edges, while
maintaining the degree sequence. See Figure~\ref{fig:switch}.  We will
return to the switch chain in Section~\ref{s:switch}.

\smallskip

\begin{figure}[ht!]
\begin{center}
\begin{tikzpicture}[scale=1.0]
\draw [fill] (0,0) circle (0.1); \draw [fill] (0,2) circle (0.1);
\draw [fill] (2,0) circle (0.1); \draw [fill] (2,2) circle (0.1);
\draw [thick,-] (0,0) -- (0,2);  \draw [thick,-]  (2,0) -- (2,2);
\draw [thick,-,dashed] (0,0) -- (2,0);  \draw [thick,-,dashed]  (0,2) -- (2,2);
\draw [<->, line width=3pt] (3.5,1) -- (4.5,1);
\begin{scope}[shift={(6,0)}]
\draw [fill] (0,0) circle (0.1); \draw [fill] (0,2) circle (0.1);
\draw [fill] (2,0) circle (0.1); \draw [fill] (2,2) circle (0.1);
\draw [thick,-] (0,0) -- (2,0);  \draw [thick,-]  (0,2) -- (2,2);
\draw [thick,-,dashed] (0,0) -- (0,2);  \draw [thick,-,dashed]  (2,0) -- (2,2);
\end{scope}
\end{tikzpicture}
\caption{Transitions of the switch chain}
\label{fig:switch}
\end{center}
\end{figure}

A Markov chain needs a starting state: that is, we must be able
to initially construct a single instance of $G(\kvec)$.  For graphs,
bipartite graphs and directed graphs, if the degree sequence is graphical
then a realization of that degree sequence can easily be constructed.
This is done using the Havel--Hakimi algorithm~\cite{hakimi,havel} for graphs,
Ryser's algorithm~\cite{ryser} for bipartite graphs and an
adaptation of the Havel--Hakimi algorithm for directed graphs~\cite{EMT-HH}.
We remark that the situation for hypergraphs is more complicated, as the
existence problem (``Does a given degree sequence have a realization?'')
is NP-complete for 3-uniform hypergraphs~\cite{DLMO}.

In theoretical computer
science, any polynomial runtime is seen as efficient.
In practice, of course, an algorithm with a high-degree polynomial
runtime may be too slow to use.  All algorithms discussed in previous
sections run until some natural stopping time is reached: that is, by
looking at the current state we can tell whether or not the algorithm
may successfully halt.  In MCMC sampling, however, 
the user must specify the number of transitions $T$ that the Markov chain 
will perform before producing any output.  Typically $T$ is defined
to be the best-known upper bound on the mixing time $\tau(\varepsilon)$,
for a suitable tolerance $\varepsilon$ (see Definition~\ref{def:mixing}).  
For this reason,
loose upper bounds on the mixing time have a significant impact on
the runtime.

Most Markov chain approaches to sampling from $G(\kvec)$ have been analysed 
using the \emph{multicommodity flow} method of Sinclair~\cite{sinclair},
which we describe in Section~\ref{s:flow}. 
Unfortunately, it is often very difficult to obtain tight 
bounds on the rate of mixing time of a Markov chain
using the multicommodity flow method.   
In any case, it is an interesting theoretical challenge to try to
characterise families of degree sequences $\kvec$ for which natural
Markov chains on $G(\kvec)$  have polynomial mixing time.

We now introduce some necessary Markov chain
background.  For more information see for example~\cite{jerrum-book,LPW}.

\subsection{Markov chain background }\label{s:markov-background}

A time-homogenous \emph{Markov chain} $\mathcal{M}$
on a finite state space $\Omega$ is a
stochastic process $X_0,X_1,\ldots$ such that $X_t\in\Omega$ for all
$t\in\N$, and 
\[ \Pr(X_{t+1}=y\mid X_0 = x_0,\ldots, X_t=x_t) = \Pr(X_{t+1}=y\mid X_t = x_t)\]
for all $t\in\N$ and $x_0,\ldots, x_t,y\in\Omega$.  
These probabilities are stored in an $|\Omega|\times |\Omega|$ matrix $P$,
called the \emph{transition matrix} $P$ of $\mathcal{M}$.
(This notation clashes with our earlier use of $P$ for configurations,
but we will not mention configurations in this section so this should
not cause confusion.)
So the $(x,y)$ entry of $P$, denoted $P(x,y)$, is defined by
\[ P(x,y) = \Pr(X_{t+1}=y\mid X_t=x)\]
for all $x,y\in\Omega$ and all $t\geq 0$.

A Markov chain is \emph{irreducible} if there is a sequence of transitions
which transforms $x$ to $y$, for any $x,y\in\Omega$,  and it is 
\emph{aperiodic} if $\gcd\{ t\mid P^t(x,x)>0\} = 1$ for all $x\in\Omega$.  
If a Markov chain is irreducible and aperiodic then we say it is 
\emph{ergodic}. 
The classical theory of Markov chains says that if $\mathcal{M}$
is ergodic then it has a unique stationary distribution
which is the limiting distribution of the chain.

We say that the Markov chain $\mathcal{M}$ is \emph{time-reversible}
(often just called \emph{reversible}) with respect to the distribution $\pi$
on $\Omega$ if the \emph{detailed balance} equations hold:
\[ \pi(x) P(x,y) = \pi(y) P(y,x)\]
for all $x,y\in\Omega$.  If a Markov chain $\mathcal{M}$ is ergodic
and is time-reversible with respect to a distribution $\pi$, then $\pi$
is the (unique) stationary distribution of $\mathcal{M}$. 
(See for example~\cite[Proposition~1.19]{LPW}.)
In particular, if $P$ is symmetric then the stationary distribution
is uniform.  The detailed balance equations are often used to guide the
design of the transition matrix of a Markov chain, so that it has the
desired stationary distribution.

Now assume that $\mathcal{M}$ is ergodic with stationary distribution $\pi$.
For $x\in\Omega$ let $P^t_x$ denote the distribution of $X_t$, 
conditioned on the event $X_0=x$.  
Recall the definition of total variation distance (Definition~\ref{def:dTV}).
For any initial state $x\in\Omega$, the distance 
$d_{TV}(P^t_x,\pi)$ is a geometrically-decreasing function of $t$
(see for example~\cite[Theorem 4.9]{LPW}).  This leads to the following
definition.

\begin{Def}
\label{def:mixing}
Let $\varepsilon > 0$ be a constant.
The \emph{mixing time} of the Markov
chain is the function
\[ \tau(\varepsilon) = \max_{x\in\Omega}\, \min\{ t \in\N \mid
    d_{TV}(P^t_x,\pi) < \varepsilon\}.\]
\end{Def}

Here $\varepsilon$ is a user-defined tolerance, which specifies how much
variation from the stationary distribution is acceptable.   The mixing time
captures the earliest time $t$ at which $P_x^t$ is guaranteed to be
$\varepsilon$-close to the stationary distribution, regardless of the starting
state.
Then $P^t_x$ remains $\varepsilon$-close to the stationary
distribution 
for all times $t\geq \tau(\varepsilon)$ and for every initial state $x\in\Omega$.

As usual, we are really interested in sampling from a set $\Omega_n$,
parameterised by $n$, where $|\Omega_n|\to\infty$,  
and we want to know how the runtime of the algorithm
behaves as $n\to\infty$.  The tolerance $\varepsilon=\varepsilon_n$
may also depend on $n$.  We say that the Markov
chain is \emph{rapidly mixing} if the mixing time is bounded above by
a polynomial in $\log|\Omega_n|$ and $\log(\varepsilon^{-1})$.
Normally it is prohibitively difficult to find $\tau(\varepsilon)$ exactly, 
so we aim to find an upper bound $T$ which is polynomial in
$\log|\Omega_n|$ and $\log(\varepsilon^{-1})$.  
Then the Markov chain can be used as an FPAUS for sampling from
$\Omega_n$, in the sense of Definition~\ref{def:FPAUS}, as follows:
starting from a convenient initial state, run the Markov chain for
$T$ steps and output the state $X_T$.  See for example 
Figure~\ref{fig:switch-code}.

We defer introduction of the multicommodity flow method until
Section~\ref{s:flow}.

\subsection{The Jerrum--Sinclair chain }\label{s:jerrum-sinclair}

The first Markov chain algorithm for sampling from $G(\kvec)$
was given by Jerrum and Sinclair~\cite{JS90} in 1990.
They used Tutte's construction~\cite{T1954} to reduce the problem
to that of sampling perfect and near-perfect matchings from an auxiliary
graph $\Gamma(\kvec)$,  then applied their Markov chain 
from~\cite{JS-approx} to solve this problem.  
This resulted in a Markov chain which has uniform stationary 
distribution over the expanded state space
$G'(\kvec) = \cup_{\kvec'} \, G(\kvec')$,
where the union is taken over the set of all graphical sequences
$\kvec'=(k'_1,\ldots, k'_n)$ such that 
$k'_j\leq k_j$ for all
$j\in [n]$ and $\sum_{j\in [n]} |k_j-k'_j| \leq 2$.

The chain performs three types of transitions, which when
mapped back to $G'(\kvec)$ are as follows: deletion of a random edge,
if the current state belongs to $G(\kvec)$; insertion of an edge
between the two distinct vertices with degree deficit one; or 
insertion of a random edge $\{i,j\}$ together with the deletion
of a randomly-chosen neighbouring edge $\{j,\ell\}$.  (We will not
specify the transition probabilities precisely here.)  See Figure~\ref{fig:JS},
where dashed lines represent non-edges.
The third type of transition is called a \emph{hinge-flip} by
Amanatidis and Kleer~\cite{AK}, following~\cite{CAR}.
\begin{figure}[ht!]
\begin{center}
\begin{tikzpicture}[scale=0.85]
\draw [fill] (0,0) circle (0.1); \draw [fill] (0,2) circle (0.1);
\draw [fill] (2,0) circle (0.1); \draw [fill] (2,2) circle (0.1);
\draw [thick,-] (0,0) -- (0,2);  \draw [thick,-,dashed]  (2,0) -- (2,2);
\node [left] at (-0.1,0) {$j$}; \node [left] at (-0.1,2) {$i$};
\node [right] at (2.1,0) {$j$}; \node [right] at (2.1,2) {$i$};
\draw [<->, line width=3pt] (0.5,1) -- (1.5,1);
\begin{scope}[shift={(5,0)}]
\draw [fill] (1,0) circle (0.1); \draw [fill] (0,2) circle (0.1);
\draw [fill] (2,2) circle (0.1); 
\draw [thick,-] (2,2) -- (1,0);  \draw [thick,-,dashed]  (0,2) -- (1,0);
\node [left] at (0.9,0) {$j$}; \node [left] at (-0.1,2) {$i$};
\node [right] at (2.1,2) {$\ell$};
\draw [<->, line width=3pt] (2.5,1) -- (3.5,1);
\begin{scope}[shift={(4,0)}]
\draw [fill] (1,0) circle (0.1); \draw [fill] (0,2) circle (0.1);
\draw [fill] (2,2) circle (0.1); 
\draw [thick,-,dashed] (2,2) -- (1,0);  \draw [thick,-]  (0,2) -- (1,0);
\node [left] at (0.9,0) {$j$}; \node [left] at (-0.1,2) {$i$};
\node [right] at (2.1,2) {$\ell$};
\end{scope}
\end{scope}
\end{tikzpicture}
\caption{Transitions of the Jerrum--Sinclair chain: insertion/deletion (left) and hinge-flip (right)}
\label{fig:JS}
\end{center}
\end{figure}

We can use the Jerrum--Sinclair chain to repeatedly 
sample from $G'(\kvec)$ until an element of $G(\kvec)$ is obtained.
For this to be efficient, the expected number of iterations required must
be bounded above by a polynomial.

\begin{Def}
\label{def:P-stable}
A class of degree sequences is called \emph{P-stable} if there
exists a polynomial $q(n)$ such that $|G'(\kvec)|/|G(\kvec)|\leq q(n)$
for every degree sequence $\kvec = (k_1,\ldots, k_n)$ in the class.
\end{Def}

Jerrum and Sinclair proved the following result~\cite[Theorem~2.4]{JS90},
but did not give an explicit (polynomial) bound on the mixing time
of their chain.

\begin{theorem} \emph{\cite{JS90}}\
There is an FPAUS for $G(\kvec)$ for any degree sequence
$\kvec$ which belongs to some P-stable class.
\label{thm:JS}
\end{theorem}

Various classes of degree sequences are known to be P-stable,
including the class of all regular sequences, and all sequences
with $\kmax$ sufficiently small.  We discuss
P-stability further in Section~\ref{s:stability}.

\subsubsection{A complete solution for the bipartite case}

The Jerrum--Sinclair chain for sampling perfect matchings
from a given graph~\cite{JS-approx} is
slow when the ratio of the number of perfect matchings to the number
of near-perfect matchings is exponentially small.
In 2004, Jerrum, Sinclair and Vigoda~\cite{JSV} described and analysed 
an ingenious algorithm, based on simulated annealing, which overcame
this problem for bipartite graphs.  Their algorithm
gives an FPAUS (and hence an FPRAS) for 
approximately-uniformly sampling (or approximately counting) perfect
matchings from a given bipartite graph. 
An important idea in~\cite{JSV} is to use a non-uniform stationary
distribution over the set of all perfect and near-perfect matchings,
so that the stationary probability of the set of perfect matchings
is at least $1/(4n^2+1)$;  that is, at most polynomially small.
This is achieved by assigning weights to each ``hole pattern''
(for a near-perfect matching, this is the pair of vertices with
deficit one, and for a perfect matching this is the empty set),
as well as edge weights.  Estimating good values for the weights
is achieved iteratively, using simulated annealing.

As a corollary, 
using Tutte's construction~\cite{T1954}, Jerrum, Sinclair and Vigoda obtained
an FPAUS for sampling bipartite graphs with given 
degrees~\cite[Corollary~8.1]{JSV}.
In fact their result is more general: given an arbitrary bipartite
subgraph $H$, they obtain an FPAUS for sampling subgraphs of $H$
with a given degree sequence.  

\begin{theorem} \emph{\cite[Corollary~8.1]{JSV}}\
\label{thm:JSV}
Given an arbitrary bipartite graph $H$, there is an FPAUS
for the set of labelled subgraphs of $H$ with a specified
degree sequence, and there is an FPRAS for computing
the number of these subgraphs.
\end{theorem}

Bez{\'a}kov{\'a}, Bhatnagar and Vigoda \cite{BBV} gave
a more direct implementation of the algorithm from~\cite{JSV},
which avoids Tutte's construction.  This allows them to obtain faster runtime
bounds compared with~\cite{JSV}.
It follows from the proof of~\cite[Theorem~1]{BBV} that their FPAUS 
is valid for any bipartite degree sequence, and has running time
\[ O\big((n_1 n_2)^2 M^3 \kmax \, \log^4(n_1n_2/\varepsilon)\big), \]
where $n_1$ and $n_2$ are the number of nodes in each part of the bipartition,
and, as usual, $k_{\max}$ is the maximum degree and $M$ is the
sum of the degrees.

\subsection{The multicommodity flow method}\label{s:flow}

There are a few methods for bounding the mixing times of Markov chains.
Before proceeding further we describe
the multicommodity flow method, which has been used
to analyse most MCMC algorithms for sampling graphs.
For information on other methods for bounding the mixing times of Markov chains,
see for example~\cite{DG-survey,jerrum-book,LPW,MT}.

Let $\mathcal{M}$ be a time-reversible ergodic Markov chain
and let $N=|\Omega|$ be the cardinality of the state space.
Then the eigenvalues of the transition matrix are real and satisfy
\[ 1 = \lambda_0 > \lambda_1 \geq \cdots \geq \lambda_{N-1} > -1.\]
The mixing time of the Markov chain is controlled by 
$\lambda_{\max} = \max\{ \lambda_1, |\lambda_{N-1}|\}$. 
Denote the smallest stationary probability by 
$\pi^*=\min\{ \pi(x)\mid x\in\Omega\}$.  Then
\begin{equation}
\label{eq:eigenvalues-mixing}
 \tau(\varepsilon) \leq (1-\lambda_{\max})^{-1}\,
    \log\bigg(\frac{1}{\varepsilon\, \pi^*}\bigg),
\end{equation}
see for example~\cite[Proposition~1]{sinclair}).

If $\lambda_{\max} = |\lambda_{N-1}|$ then in particular,
$\lambda_{N-1}$ must be negative, 
in which case $1-|\lambda_{N-1}| = 1+\lambda_{N-1}$.
For many chains we can apply a result of Diaconis and 
Saloff-Coste~\cite[p.702]{DSC} (see also~\cite{lazy}) to establish
an upper bound on $(1+\lambda_{N-1})^{-1}$.   
In particular, if there is a positive
probability of a null transition at any state then the following
special case of that result may be useful:
\begin{equation}
\label{convenient}
 (1 + \lambda_{N-1})^{-1} \leq \nfrac{1}{2}\, \max_{x\in\Omega}
   P(x,x)^{-1}.
\end{equation}
Another option is to work with the
\emph{lazy} version of the Markov chain $\mathcal{M}$,
by replacing the transition matrix $P$ by $(I+P)/2$. 
This ensures that all eigenvalues
are nonnegative and hence $\lambda_{\max}=\lambda_1$.
We say that a Markov chain $\mathcal{M}$ \emph{is lazy} if
$P(x,x)\geq \nfrac{1}{2}$ for all $x\in\Omega$.

Sinclair's \emph{multicommodity flow method}~\cite{sinclair}
provides an upper bound on $(1-\lambda_1)^{-1}$.
It is a generalisation of the \emph{canonical path} method
that Jerrum and Sinclair introduced in~\cite{JS-approx}.

Given a Markov chain $\mathcal{M}$ with uniform stationary distribution
on a state space $\Omega$, let $\mathcal{G}(\mathcal{M})$
be the underlying graph, where there is an edge from
$x$ to $y$ if and only if $P(x,y)>0$.
We assume that $\mathcal{M}$ is ergodic and time-reversible
with respect to the distribution $\pi$.  
Let $\mathcal{P}_{xy}$ be the set of all simple directed paths
from $x$ to $y$ in $\mathcal{G}(\mathcal{M})$, and define $\mathcal{P}
= \cup_{x,y} \mathcal{P}_{xy}$.   A \emph{flow} is a function
$f:\mathcal{P}\to [0,\infty)$ such that for all $x,y\in\Omega$
with $x\neq y$,
\[ \sum_{p\in\mathcal{P}_{xy}} f(p) = \pi(x)\pi(y).\]
(In the canonical path method,
there is only one flow-carrying path from between any two pairs of states.)

If the flow can be defined so that no transition of the chain 
is overloaded,
then the state space does not contain any ``bottlenecks''
and the Markov chain will be rapidly mixing.
To make this precise, the total flow through a transition $e=xy$ is
$f(e) = \sum_{p\ni e} f(p)$,   and the \emph{load} of 
$e$ is defined by $\rho(e) = f(e)/Q(e)$, where $Q(e) = \pi(x) P(x,y)$
is the \emph{capacity} of the transition $e=xy$.  (By time-reversibility,
$Q$ is well-defined.)
Finally, the  maximum load of the flow is
$\rho(f) = \max_e \rho(e)$, while $\ell(f)$ denotes the 
length of the longest path $p$ with $f(p)>0$.  
Sinclair~\cite[Corollary 6']{sinclair} proved that for any time-reversible
Marvov chain any any flow $f$,
\begin{equation}
\label{eq:congestion-bound}
 (1-\lambda_1)^{-1} \leq \rho(f)\,\ell(f).
\end{equation}

The next result specialises the multicommodity flow method to
ergodic, time-reversible Markov chains with the uniform stationary 
distribution over a set $\Omega$. It is obtained from 
(\ref{eq:eigenvalues-mixing}) and  (\ref{eq:congestion-bound}),
and allows two options for managing the smallest eigenvalue
$\lambda_{N-1}$.  

\begin{theorem} 
\label{thm:flow}
Let $\mathcal{M}$ be an ergodic time-reversible Markov chain with
uniform stationary distribution over $\Omega$.
Define $B$ to be 0, if it is known that $\lambda_{\max}=\lambda_1$ (for example
if $\mathcal{M}$ is lazy).  Otherwise, let $B$ be an upper bound 
on $(1+\lambda_{N-1})^{-1}$.
Then the mixing time of the Markov chain $\mathcal{M}$ satisfies
\[ \tau(\varepsilon) \leq \max\{\rho(f) \ell(f),B\} \,
    \big(\log |\Omega| + \log(\varepsilon^{-1})\big).\]
\end{theorem}

When the multicommodity flow method is applied to the problem of
sampling graphs from $G(\kvec)$, the start and end states are
graphs $G,G'$ with degree sequence $\kvec$.  Usually the definition 
of the flow is guided by the symmetric difference $H=G\triangle G'$
of $G$ and $G'$, and each step of a flow-bearing path is designed
to make the symmetric difference smaller.
\subsection{The switch chain}\label{s:switch}

The \emph{switch chain} (also called \emph{swap chain}~\cite{MES}
and \emph{Diaconis chain}~\cite{BBV})
is the simplest Markov chain
with uniform distribution over $G(\kvec)$.  A transition of the switch chain
deletes two edges and inserts two edges,
while maintaining the degree sequence and without
introducing any repeated edges.  This is illustrated in Figure~\ref{fig:switch}, at the
start of Section~\ref{s:markov}.
This chain was introduced by Diaconis and Gangolli~\cite{DG} in 1995
in order to sample contingency tables 
(matrices of nonnegative integers) with given row and column sums.
The transitions can be easily adapted to bipartite graphs or directed graphs.
The switch chain is ergodic for graphs, and for bipartite graphs,
with given degrees.

In 1999, Kannan, Tetali and Vempala~\cite{KTV} considered the switch chain
for sampling bipartite graphs with given degrees.  They used
the auxiliary graph $\Gamma(\kvec)$ obtained
from Tutte's construction (modified to bipartite graphs) in order
to define a multicommodity flow, and gave details only for the case
of regular degrees.
Unfortunately, there is a bug in their argument (specifically, in the proof
of~\cite[Theorem~4.1]{KTV}) which seems to be fatal.\footnote{The 
symmetric difference of two perfect matchings in $\Gamma(\kvec)$ consists
of the union of disjoint cycles.
However, when mapped back to the symmetric difference of the
two corresponding bipartite graphs, an alternating cycle in $\Gamma(\kvec)$  
may correspond to an alternating \emph{walk}, which could have
linearly many repeated vertices (vertices which are visited more than once on the walk).  The argument of~\cite{KTV} does not take this into account.}

\bigskip

\begin{figure}[ht!]
\begin{center}
\hspace{\dimexpr-\fboxrule-\fboxsep\relax}\fbox{
\begin{minipage}{0.8\textwidth}
\begin{tabbing}
  \textsc{The switch chain }\\
    X \= \kill
  \> XXXs \= \kill 
  \emph{Input:}\>\> graphical sequence $\kvec$ and positive integer $T=T(\kvec)$\\
  \emph{Output:}\>\> element of $G(\kvec)$\\
  \\
    let $G$ be an arbitrary initial state\\
    for $t=0,\ldots, T-1$ do\\
    \> choose two non-adjacent distinct edges $\{ a,b\}, \{ c,\, d\}$ u.a.r. \\
    \>  choose a perfect matching $M$ of $\{a,b,c,d\}$ u.a.r.\\
    X \= \kill
    \>  if $M\cap E(G) = \emptyset$ then\\
    \> X \= \kill
    \> \> delete the edges $\{a,b\}, \{ c,d\}$ and add the edges of $M$\\
    output $G_T$
\end{tabbing}
\end{minipage}}
\caption{The switch chain for sampling from $G(\kvec)$}
\label{fig:switch-code}
\end{center}
\end{figure}

One implementation of the switch chain for $G(\kvec)$ is given in
Figure~\ref{fig:switch-code}.
Cooper, Dyer and Greenhill~\cite{CDG} analysed the mixing time of 
the lazy version of this chain, restricted to regular degree sequences.
(That is, they replaced the transition matrix $P$ arising from the above 
procedure by $(I+P)/2$, which is equivalent to inserting the instruction
``With probability $\nfrac{1}{2}$, do nothing'' just inside the for-loop.)
However, this is unnecessary, as the transition 
procedure given in Figure~\ref{fig:switch-code} guarantees that 
$(1+\lambda_{\max})^{-1} \leq \nfrac{3}{2}$,  by (\ref{convenient}). 
Hence we can take $B=\nfrac{3}{2}$ in Theorem~\ref{thm:flow}.
Using a multicommodity flow argument,
Cooper et al.~established a polynomial bound on the mixing time
for any regular degree sequence~\cite{CDG,CDG-corrigendum}.

\begin{theorem} \emph{\cite{CDG,CDG-corrigendum}}\
For any $k=k(n)\geq 3$, the switch chain on $G(n,k)$ has
mixing time
\[ \tau(\varepsilon)\leq k^{23}\, n^8\, \big( kn\log(kn) + \log(\varepsilon^{-1})\big).
\]
\end{theorem}

\begin{proof}  (Sketch.)\
Cooper et al.\ defined a multicommodity flow for the switch
chain on $G(n,k)$ and proved that maximum load $\rho(f)$ 
is bounded above by a polynomial in $n$ and $k$.   
(A brief outline of the argument is given below.)
The length of any flow-carrying path is at most $kn/2$.
Next,
\begin{equation}
\label{Gk-size}
   |G(\kvec)| \leq \frac{M!}{2^{M/2}\, (M/2)!\, \prod_{j\in [n]} k_j!}
  \leq \exp(\nfrac{1}{2} M\log M\big),
\end{equation}
where the first inequality follows from the configuration model.
Applying Theorem~\ref{thm:flow} completes the proof. 
\end{proof}

Greenhill~\cite{directed} used a similar argument to show that the switch chain
for $k$-in, $k$-out (regular) directed graphs is rapidly mixing 
for any $k$ with $1\leq k=k(n)\leq n-1$ and all $n\geq 4$.

\bigskip

We now give some more details on the design and analysis of the 
multicommodity flow for the switch chain.
The flow between two graphs $G,G'$ is defined with respect to the 
symmetric difference $H=G\triangle G'$.  Note that the
symmetric difference $H$ need not be regular, even if $G,G'$ are both regular.
Greenhill and Sfragara~\cite{GS} observed that the multicommodity flow
defined in~\cite{CDG} for regular degrees can also be used for 
irregular degrees.
In fact, almost all parts of the analysis of the multicommodity
flow also extends immediately to irregular degree sequences.
For this reason, the description below is presented in the
general setting of $G(\kvec)$.

Starting from the symmetric difference $H=G\triangle G'$,
Cooper et al.\ described how to decompose this symmetric difference
into a sequence of smaller, edge-disjoint structures they called
1-circuits and 2-circuits.  They identified several different ways to
do this, parameterised by a set $\Psi(G,G')$. 
For each $\psi\in\Psi(G,G')$ , the (canonical) path from $G$ to $G'$
indexed by $\psi$ is denoted $\gamma_\psi(G,G')$. 
This path is created by ``processing'' each of the 1-circuits and 2-circuits,
in a specified order.  Processing a circuit changes the status of its edges
from agreeing with $G$ to agreeing with $G'$, and adds some transitions
to the path $\gamma_\psi(G,G')$.  Finally, the flow from $G$
to $G'$ divided equally among these $|\Psi(G,G')|$ paths.

Once the multicommodity flow is defined, it remains to prove that 
no transition is too heavily loaded.  Suppose that $e=(Z,Z')$ is a
transition of the chain which is used on the path $\gamma_\psi(G,G')$.
A common approach is to define an \emph{encoding} $L$ of a state $Z$, which 
records information about the symmetric difference $H=G\triangle G'$ to help
us recover $G$ and $G'$ from the transition $(Z,Z')$ and the encoding $L$.
This approach will work if the set of possible encodings for $Z$ is at most
polynomially larger than $|G(\kvec)|$, and there are at most polynomially-many
options for $(G,G')$ once $Z,Z',L$ are all specified.
In~\cite{CDG}, encodings are defined by
\[ L + Z = G + G'\]
where $G, G'$ and $Z$ are identified with their $n\times n$ adjacency 
matrices.   Then $L$ is an $n\times n$ symmetric matrix with row sums
given by $\kvec$, and with almost all entries equal to 0 or 1.
In fact, due to the careful way that 1-circuits and 2-circuits are
processed, $L$ has at most 4 entries equal to $-1$ or 2, and all
other entries are 0 or 1.  Hence $L$ may also be thought of as a graph
with most edges labelled 1, and at most four \emph{defect edges}
which may be labelled $-1$ or 2.  The sum of all edge labels at vertex $j$ 
must equal the degree $k_j$ of $j$, for all $j\in [n]$. 
(There are some other constraints
about the structure of the defect edges, stated in~\cite[Lemma~2]{CDG}.)

Given an encoding $L$, by removing the defect edges we obtain
a graph with degree sequence which is very close to $\kvec$.
This gives a connection between the ratio $|\mathcal{L}(Z)|/|G(\kvec)|$ 
and the ratio $|G'(\kvec)|/|G(\kvec)|$ from the definition of
P-stability (\ref{def:P-stable}).
This connection is explored further in Section~\ref{s:stability}.
In the regular case, we have the following bound.

\begin{lemma} \emph{\cite[Lemma~4]{CDG}}\
\label{lem:CDG-encodings}
For any $Z\in G(n,k)$, let $\mathcal{L}(Z)$ denote the set of encodings $L$ 
such that every entry of $L+Z$ belongs to $\{0,1,2\}$. 
Then $|\mathcal{L}(Z)|\leq 2 k^6 n^6\, |G(n,k)|$.
\end{lemma}

\begin{proof} (Sketch.)\
This was proved by extending the switch operation to encodings, and showing 
that at most three switches suffice to transform an encoding into an
element of $G(n,k)$. The factor $2 k^6 n^6$ is an upper bound on the
number of encodings which can be transformed into an arbitrary element of
$G(n,k)$.
\end{proof}

In fact, Greenhill and Sfragara observed that~\cite[Lemma~4]{CDG} 
(restated as Lemma~\ref{lem:CDG-encodings} above) was
the \emph{only} part of the argument from~\cite{CDG} which relied on the 
regularity assumption for its proof.
For this reason, they called it the ``critical lemma''.
The proof of the lemma is essentially a switching argument, used
to find the relative sizes of two sets.  For irregular degrees, it was
no longer possible to prove that a suitable switch could always be
found in any encoding.

Greenhill and Sfragara bypassed this problem by using a more powerful
switching operation to prove the critical lemma.  Rather than use a switch,
which swaps edges for non-edges around an alternating 4-cycle,
they used a operation involving an alternating 6-cycle (deleting 3 edges and
inserting 3 edges at a time), illustrated in Figure~\ref{fig:6-switch}.  
This operation gave them sufficient 
flexibility to prove the critical lemma when $k_{\max}$ is not too large.  

\medskip

\begin{figure}[ht!]
\begin{center}
\begin{tikzpicture}[scale=0.85]
\draw [fill] (0,0) circle (0.1); \draw [fill] (0,2) circle (0.1);
\draw [fill] (2,0) circle (0.1); \draw [fill] (2,2) circle (0.1);
\draw [fill] (4,0) circle (0.1); \draw [fill] (4,2) circle (0.1);
\draw [thick,-] (0,0) -- (0,2);  \draw [thick,-]  (2,0) -- (4,0);
\draw [thick,-,dashed] (0,0) -- (2,0);  \draw [thick,-,dashed]  (4,0) -- (4,2);
\draw [thick,-] (4,2) -- (2,2);  \draw [thick,dashed,-]  (2,2) -- (0,2);
\draw [<->, line width=3pt] (5.5,1) -- (6.5,1);
\begin{scope}[shift={(8,0)}]
\draw [fill] (0,0) circle (0.1); \draw [fill] (0,2) circle (0.1);
\draw [fill] (2,0) circle (0.1); \draw [fill] (2,2) circle (0.1);
\draw [fill] (4,0) circle (0.1); \draw [fill] (4,2) circle (0.1);
\draw [thick,dashed,-] (0,0) -- (0,2);  \draw [thick,dashed,-]  (2,0) -- (4,0);
\draw [thick,-] (0,0) -- (2,0);  \draw [thick,-]  (4,0) -- (4,2);
\draw [thick,-,dashed] (4,2) -- (2,2);  \draw [thick,-]  (2,2) -- (0,2);
\end{scope}
\end{tikzpicture}
\caption{Switching edges around a 6-cycle}
\label{fig:6-switch}
\end{center}
\end{figure}

\begin{theorem}
\emph{\cite[Theorem 1.1]{GS}}\
Let $\kvec$ be a graphical degree sequence. 
If all entries of $\kvec$ are positive and $3\leq k_{\max}\leq \nfrac{1}{3}\sqrt{M}$ then the mixing time of 
the switch chain on $G(\kvec)$ satisfies
\[ \tau(\varepsilon) \leq k_{\max}^{14}\, M^9 \big( \nfrac{1}{2} M\log M
   + \log(\varepsilon)\big).\]
\label{thm:GS}
\end{theorem}

In Section~\ref{s:stability} we will discuss 
connections between the stability of degree sequences
and rapid mixing of the switch chain. 
First we discuss some results regarding the switch chain
for bipartite graphs, directed graphs and hypergraphs, and some
related topics.

\subsubsection{Bipartite graphs and directed graphs}

We have seen in Theorem~\ref{thm:JSV}
that the algorithm of Jerrum, Sinclair and Vigoda~\cite{JSV}
gives an FPAUS for sampling bipartite graphs with any given
bipartite degree sequence.  However, there is still interest
in studying the switch chain for bipartite graphs, as it
is a very natural and simple process.

A 1-regular bipartite graph is a permutation. Diaconis and
Shahshahani~\cite{DShah81} studied the Markov chain with state
space $S_n$, the set of all permutations of $[n]$, and transitions
defined as follows:  with probability $1/n$ do nothing, and otherwise
choose a transposition $(i\,j)$ uniformly at random (where $i,j\in [n]$
are distinct), and multiply the current permutation by $(i\, j)$ on 
the left, say.  This random transposition chain is very closely
related to the switch chain for a 1-regular bipartite degree sequence
(the set of allowed transitions is identical, though the
probability of each transition differs between the two chains).  Diaconis and
Shahshahani gave a very complete analysis of the random transposition
chain in~\cite{DShah81}, calculating the eigenvalues and proving
that the chain exhibits the ``cutoff phenomenon'', see~\cite[Chapter~18]{LPW}.  That is,
the total variation distance to stationarity drops very quickly from
$1-o(1)$ to $o(1)$ when the chain has taken $\nfrac{1}{2} n\log n + \Theta(n)$
steps (this is cutoff at $\nfrac{1}{2} n\log n$ with window of order $n$).

Other than~\cite{DShah81}, the first analysis of the switch chain for
sampling bipartite graphs with given degrees
was the work of Kannan et al.~\cite{KTV}, discussed earlier.

The
multicommodity flow arguments from~\cite{CDG,GS} can be simplified
when restricted to bipartite graphs, as the symmetric difference of
two bipartite graphs with the same degree sequence
can be decomposed into edge-disjoint alternating cycles, and these
are relatively easy to handle.  The resulting bounds on the mixing
time of the switch chain for bipartite graphs with given degrees were
recently presented in~\cite[Appendix~A]{DGKRS}.
These show that the switch chain is rapidly mixing for any regular
bipartite degree sequence, and for arbitrary bipartite degree sequences
when the maximum degree is not too large compared to the number of edges.
As usual, the mixing time bounds are very high-degree polynomials.

A bipartite degree sequence is \emph{half-regular} if all degrees
on one side of the bipartition are regular.
Miklos, Erd\H{o}s and Soukup \cite{MES} proved 
that the switch Markov chain is rapidly
mixing for half-regular bipartite degree sequences. 
Their proof also used the multicommodity flow method, but the
flow is defined differently to the Cooper--Dyer--Greenhill flow described
above.   

For some directed degree sequences, the switch chain fails to
connect the state space, as it cannot reverse the orientation of
a directed 3-cycle.  
Rao et al.~\cite{RJB} observed that by the Markov chain which
performs switch moves and (occasionally) reverse directed 3-cycles,
is ergodic for any directed degree sequence.  They noted that
for many degree sequences, this additional move did not seem to be
needed in order to connect the state space.  This was confirmed by the 
work of 
Berger and M{\" u}ller-Hannemann~\cite{BM-H} and LaMar~\cite{lamar},
who characterised degree sequences for which the switch chain is
irreducible.

Greenhill and Sfragara~\cite[Theorem 1.2]{GS} adapted their argument
to directed graphs, proving a similar result to Theorem~\ref{thm:GS}.
As well as an upper bound on the maximum degree,~\cite[Theorem 1.2]{GS}
also assumes that the switch chain connects the state space.
Their argument built on 
Greenhill's analysis~\cite{directed} of the switch chain for directed graphs,
replacing the proof of the ``critical lemma'' from~\cite{directed}
by one which did not require regularity.  

\subsubsection{The augmented switch chain and the Curveball chain}

Erd\H{o}s et al.~\cite{EKMS} considered the
switch chain augmented by an additional transition, namely switching
the edges around an alternating 6-cycle as shown in Figure~\ref{fig:6-switch}.
They called this transition a \emph{triple swap}.   We will refer to this chain as
the \emph{augmented switch chain}.
Building on the analysis from~\cite{MES}, 
Erd\H{o}s et al.~\cite[Theorem~10]{EKMS} proved that
the augmented switch chain for half-regular bipartite degree sequences remains
rapidly mixing in the presence of set of forbidden edges given by
the union of a perfect matching and a star.
They also described an algorithm (similar to the Havel--Hakimi algorithm)
for constructing a single
realization~\cite[Theorem~9]{EKMS}, to be used as the initial state.
Since directed graphs can be modelled as bipartite graphs with a forbidden
perfect matching, their algorithm also gives an FPAUS for directed
graphs with specified in-degrees and out-degrees, where (say)
the sequence of in-degrees is regular.  This explains the
addition of the triple swap transitions, without which the
chain might not be irreducible for some directed degree sequences.
By avoiding a star, the problem becomes \emph{self-reducible}~\cite{JVV}, 
which leads to an FPRAS for approximating
the number of bipartite graphs with given half-regular degree sequence
and some forbidden edges. 
As mentioned earlier, the algorithm
of Jerrum, Sinclair and Vigoda~\cite[Corollary 8.1]{JSV} can also
be applied to this problem.

Erd\H{o}s et al.~\cite{EMMS} gave 
new conditions on bipartite
and directed degree sequences which guarantee rapid mixing of the
augmented switch chain.  In particular, suppose that a bipartite
degree sequence has degrees $\svec=(s_1,\ldots, s_a)$ in one part
and degrees $\tvec=(t_1,\ldots, t_b)$  in the other, where $a+b=n$.  Let 
$s_{\max}, s_{\min}, t_{\max}, t_{\min}$ be the maximum and minimum degrees
on each side.  If all degrees are positive and
\[ (s_{\max} - s_{\min}-1)(t_{\max}-t_{\min}-1)\leq \max\{ 
    s_{\min}\, (a-t_{\max}),\, t_{\min}\, (b - s_{\max})\}\]
then the augmented switch chain on the set of bipartite graphs
with bipartite degree sequence $(\svec,\tvec)$ is rapidly
mixing~\cite[Theorem 3]{EMMS}.  They applied this result to the analysis
of the bipartite Erd\H{o}s--R{\' e}nyi model $\mathcal{G}(a,b,p)$, 
with $a$ vertices in one side of the bipartition, $b$ vertices on the other
and each possible edge between the two parts is included with probability $p$.
Erd\H{o}s et al.~\cite[Corollary 13]{EMMS}  proved that
if $p$ is not too close to~0 or~1 then the augmented switch chain
is rapidly mixing for the degree sequence arising from $\mathcal{G}(a,b,p)$,
with high probability as $n\to\infty$, where $n=a+b$.
They also proved analogous results for
directed degree sequences~\cite[Theorem~4, Corollary~14]{EMMS}.  
To prove their results, they adjusted the multicommodity argument
from~\cite{EKMS,MES} and gave new proofs of the ``critical lemma''
for that argument.

The Curveball chain, introduced by Verhelst~\cite{Verhelst},
is another Markov chain for sampling bipartite
graphs with given degrees, which chooses two vertices on one
side of the bipartition and randomises their neighbourhoods,
without disturbing the degrees or set of
common neighbours of the chosen vertices.
Carstens and Kleer~\cite{CK2018} showed that the
Curveball chain is rapidly mixing whenever the switch chain is
rapidly mixing.

\subsubsection{New classes from old}

Erd\H{o}s, Mikl\'os and Toroczkai~\cite{EMT2018} described a novel
way to expand the class of degree sequences (and bipartite degree
sequences, and directed degree sequences) for which the switch
chain is known to be rapidly mixing.  Their approach utilised
canonical degree sequence decompositions, introduced by
Tyshkevich~\cite{Tysch}, and extended this concept to bipartite
and directed degree sequences.  Using the decomposition theorem
from~\cite{EMT2015}, if the switch chain (or augmented switch chain)
is rapidly mixing on each component of this decomposition, then
it is rapidly mixing on the original degree sequence.

\subsubsection{Improved bounds using functional inequalities}

Functional inequalities can be used to give tight bounds on
the convergence of of Markov chains.   Suppose that $\mathcal{M}$
is a Markov chain with state space $\Omega$, transition matrix $P$ and 
stationary distribution $\pi$.  The
\emph{Dirichlet form} associated to $\mathcal{M}$ is defined by
\[ \mathcal{E}_{P,\pi}(f,f) = \dfrac{1}{2} \sum_{x,y\in\Omega} \big(f(x)-f(y)\big)^2\, \pi(x) P(x,y)\]
for any $f:\Omega\to\R$.
This is a weighted measure of how much $f$ varies over pairs of states which 
differ by a single transition.  The \emph{variance} of $f$ with respect
to $\pi$, defined by
\[ \operatorname{Var}_\pi(f) = \dfrac{1}{2} \sum_{x,y\in\Omega} \big(f(x)-f(y)\big)^2\, \pi(x) \pi(y),\]
captures the global variation of $f$ over $\Omega$.
These functions can be used to bound the second-largest eigenvalue $\lambda_1$
of $\mathcal{M}$ as follows:
\[ 1-\lambda_1 = \min_f \, \frac{\mathcal{E}_{P,\pi}(f,f)}{\operatorname{Var}_{\pi}(f)},\]
where the minimum is taken over all non-constant functions $f:\Omega\to\R$.  See for example~\cite[Lemma 13.12 and Remark 13.13]{LPW}.
The Markov chain satisfies a \emph{Poincar{\' e} inequality}  with
constant $\alpha$ if 
$\operatorname{Var}_{\pi}(f) \leq \alpha\, \mathcal{E}_{P,\pi}(f,f)$
for any $f:\Omega\to\R$.
The log-Sobolev inequality has a similar (but more complicated) definition
and can also be used to bound the mixing time~\cite{DSC-logSob}.

Very recently, Tikhomirov and Youssef~\cite{Tik}
proved\footnote{Subject to refereeing.} a sharp Poincar{\' e} inequality, and established a
log-Sobolev inequality, for the switch chain
for regular bipartite graphs.   
Using their Poincar{\' e} inequality, Tikhomirov and Youssef 
proved~\cite[Corollary~1.2]{Tik} that when $3\leq k\leq n^c$ for some universal
constant $c$, the 
mixing time of the switch chain on $k$-regular bipartite graphs satisfies
\[ \tau(\varepsilon) \leq C k n\big( kn\log kn + \log(2\varepsilon^{-1})\big)\]
for some universal constant $C>0$ (constants $c,C$ not explicitly stated).
This is a huge improvement on any previously-known bound.
Tikhomirov and Youssef also state the following mixing time bound, obtained
using their log-Sobolev inequality when $k\geq 3$ is a fixed constant:
\[ \tau(\varepsilon) \leq C_k\, n\log n\bigg( \log n + \log\Big(\frac{1}{2\varepsilon^2}\Big)\bigg).
\]
Here $C_k>0$ is an expression which depends only on $k$.

The proof in~\cite{Tik} is long and technical, and will likely be difficult
to generalise.  But it is exciting to see such a low-degree polynomial
bound on the mixing time of the switch chain for this non-trivial
class of bipartite degree sequences.

\subsubsection{Hypergraphs}

A hypergraph is \emph{uniform} if every edge contains the same number
of vertices.
The incidence matrix of a hypergraph can be viewed as the adjacency
matrix of a bipartite graph, with one part of the bipartition representing
the vertices of the hypergraph and the other part representing the edges.  
Conversely, a bipartite graph gives
rise to a simple hypergraph, by reversing this construction, if
all vertices on the ``edge'' side of the bipartition have distinct 
neighbourhoods.  This is needed to avoid creating a repeated edge.  
In the case of uniform hypergraphs, if the resulting hypergraph is simple 
then it arises from precisely $m!$ distinct
bipartite graphs, where $m$ is the number of edges of the hypergraph.

Hence, any algorithm for sampling bipartite graphs with a given 
half-regular degree sequence
can be transformed into an algorithm for sampling uniform hypergraphs with
given degrees, using rejection sampling.
This is explored by Dyer et al.\ in~\cite{DGKRS}.
Note that a configuration model may be defined for hypergraphs,
though it is equivalent to the corresponding bipartite configuration
model.  The configuration model only gives polynomial-time uniform sampling when the
maximum degree multiplied by the edge size is $O(\log n)$,
see~\cite[Lemma~2.3]{DGKRS}.

Chodrow~\cite{chodrow} introduced a Markov chain which works directly on
uniform hypergraphs.  A transition involves choosing two edges $e,f$ and 
deleting them, then inserting two edges $e', f'$ chosen randomly so
that the degree sequence is unchanged and $e'\cap f' = e\cap f$.
This transitions of this chain are analogous to the transitions of the
Curveball chain~\cite{CK2018} for sampling bipartite graphs.
Chodrow proved that this chain is ergodic, but did not analyse
the mixing time.  It is an open problem to determine classes of 
degree sequences and edge sizes for which this chain is rapidly mixing.

\subsection{Stability of degree sequences }\label{s:stability}

Informally, a class of degree sequences is \emph{stable} 
if $|G(\kvec)|$ varies smoothly
as $\kvec$ ranges over the class~\cite{JSM}.
Work on the connection between the stability of degree sequences and 
mixing rates of Markov chains for sampling from $G(\kvec)$
began with Jerrum and Sinclair's definition of P-stability~\cite{JS90},
stated in Definition~\ref{def:P-stable} above.  A slightly different
version of P-stability was studied by Jerrum, Sinclair and McKay~\cite{JSM}.
Let $\| \xvec\|_1 = \sum_{j\in [n]} |x_j|$ denote the 1-norm of the
vector $\xvec = (x_1,\ldots, x_n)$, and define the set $U(\kvec)$
of all degree sequences $\tilde{\kvec}$ such that
\[ \|\tilde{\kvec}\|_1 = \|\kvec\|_1 \quad \text{ and } \quad
        \| \tilde{\kvec} - \kvec\|_1 = 2.\]
Jerrum, Sinclair and McKay said that a class of degree sequences is P-stable if
there exists a polynomial $q(n)$ such that
\[ 
\left| \bigcup_{\tilde{\kvec}\in U(\kvec)} G(\tilde{\kvec}) \right|
     \leq q(n)\, |G(\kvec)| \]
for all $\kvec$ in the class.  If this holds then $\kvec$ is also
P-stable in the original sense (Definition~\ref{def:P-stable}).
Let $k_{\min}$ denote the smallest entry of $\kvec$. 
Jerrum et al.~\cite[Theorem~8.1, Theorem~8.3]{JSM} gave two sufficient 
conditions for a degree sequence to belong to a P-stable class.

\begin{theorem}
\emph{\cite{JSM}}\
Recall that $M=M(\kvec)$ is the sum of entries in the
degree sequence~$\kvec$.
\begin{itemize}
\item[\emph{(i)}]
The class of graphical degree sequences $\kvec=(k_1,\ldots, k_n)$ which satisfy
\[ (k_{\max} - k_{\min} + 1)^2 \leq 4k_{\min}\, (n-k_{\max} - 1)\]
is P-stable.  
\item[\emph{(ii)}]
The class of graphical degree sequences $\kvec=(k_1,\ldots, k_n)$ which satisfy
\begin{align*}
& (M-k_{\min}\, n)(k_{\max}\, n - M) \\
& \quad\leq  (k_{\max}-k_{\min})
  \big( (M-k_{\min}\, n)(n-k_{\max}-1) + k_{\min} (k_{\max}\, n - M)\big)
\end{align*}
is P-stable.
\end{itemize}
\label{thm:P-stable}
\end{theorem}

Jerrum et al.~\cite{JSM} listed several examples of classes of degree 
sequences which satisfy one of these sufficient conditions
and hence are P-stable, including
\begin{itemize}
\item all regular sequences;  
\item all graphical sequences with $k_{\min}\geq 1$ and $k_{\max} \leq 2\sqrt{n}-2$; 
\item all graphical sequences with $k_{\min} \geq n/9$ and $k_{\max} \leq 5n/9-1$.
\end{itemize}
Using (\ref{power-law-facts}) and recalling Definition~\ref{def:plib},
we see that the 
sufficient condition from Theorem~\ref{thm:P-stable}(i) does not cover
heavy-tailed distributions such as the
power-law distribution-bounded degree sequences with $\gamma\in (2,3)$.
The condition from Theorem~\ref{thm:P-stable}(ii) may hold in some
cases but fails whenever $M > 2k_{\min}\, n$.
The first and third examples show that a P-stable class does not have to be
sparse.  

It is possible to define classes of degree sequences which
are not P-stable but for which the switch chain is rapidly 
mixing.  Jerrum, Sinclair and McKay~\cite{JSM} illustrated this using
the degree sequence 
\[ \kvec = \kvec(n) = (2n-1, 2n-2,\ldots, n+1,n,n,n-1,\ldots, 2,1)\]
on $2n$ vertices.  There is a unique realisation of this degree
sequence, so $|G(\kvec)|=1$ and
the switch chain is trivially rapidly mixing on $G(\kvec)$.
However, $|G(\kvec')|$ is exponential in $n$, where
\[ \kvec' = \kvec'(n) = (2n-2, 2n-2,\ldots, n+1,n,n,n-1,\ldots, 2,2)\]
is obtained from $\kvec$ by decreasing the largest degree by~1 and
increasing the smallest degree by~1.
Hence the class $\{ \kvec(n) \mid n\geq 2\}$ is not P-stable.
Erd{\H o}s et al.~\cite{EGMMS} described more general classes of degree
sequences with these properties.  
So P-stability is not a necessary condition
for the switch chain to be efficient.
Rather, the standard proof techniques tend to break down when
the class of degree sequences is not P-stable.

\subsubsection{Strong stability}

Amanatidis and Kleer~\cite{AK,AK-RSA} defined a new notion of stability,
called \emph{strong stability},   which is possibly stronger than P-stability.
Recall that $G'(\kvec)$
denotes the state space of the Jerrum--Sinclair chain.
Say that graphs $G,H$ are at JS-distance $r$  if
$H$ can be obtained from $G$ using at most $r$ transitions
of the Jerrum--Sinclair chain.   Next, let $d_{JS}(\kvec)$ denote
the maximum, over all $G\in G'(\kvec)$, of the minimum distance
from $G$ to an element of $G(\kvec)$.  Then every element of
the augmented state space $G'(\kvec)$ can be transformed
into an element of $G(\kvec)$ in at most $d_{JS}(\kvec)$
transitions of the Jerrum--Sinclair chain.

\begin{Def}
A class of graphical degree sequences is \emph{strongly stable}
if there is a constant $\ell$ such that $d_{JS}(\kvec)\leq \ell$
for all degree sequences $\kvec$ in the class.
\label{def:strongly-stable}
\end{Def}

Amanatidis and Kleer proved~\cite[Proposition~3]{AK}
that every strongly stable family is P-stable. It is not known
whether the converse is also true.
The main result of~\cite{AK} is the following.

\begin{theorem}
\label{thm:AK}
\emph{\cite[Proposition~2.3 and Theorem~2.4]{AK-RSA}}\
The switch chain
is rapidly mixing for all degree sequences from a 
strongly stable family.
\end{theorem}

The proof of Theorem~\ref{thm:AK} rests on the observation 
that it is much easier to define a good
multicommodity flow for the Jerrum--Sinclair chain than for the 
switch chain.  Next, Amanatidis and Kleer prove that when $\kvec$ is strongly stable,
a good flow for the Jerrum--Sinclair chain can be transformed
into a good flow for the switch chain.   

Amanatidis and Kleer gave analogous results for bipartite degree 
sequences~\cite[Theorem~17]{AK}. Their framework provided a unified
proof of many rapid mixing results for the switch chain
for graphs, and bipartite graphs, with given degrees.
The authors of~\cite{AK} remark that their ``unification of the existing
results [...] is qualitative rather than quantitative,
in the sense that our simpler, indirect approach provides weaker polynomial
bounds for the mixing time.''

\subsubsection{Rapid mixing for P-stable degree classes}

Erd\H{o}s et al.~\cite{hungarians} defined a new
multicommodity flow for the switch chain. 
The symmetric difference is decomposed into \emph{primitive}
alternating circuits, such that no vertices is visited more than twice
on a primitive circuit, and if a vertex is visited twice then
the two occurrences are at an odd distance from each other around the
circuit. Then the primitive alternating circuits are processed
in a carefully-chosen order.  An encoding is defined (and
is called an ``auxiliary matrix'') such that it is at most
three switches away from (the adjacency matrix of) an element
of the set $G'(\kvec)$.  By definition, P-stability guarantees that
$|G'(\kvec)|\leq q(n)\, |G(\kvec)|$ for some polynomial $q(n)$.
Furthermore, there are a polynomial number of ways to choose each of the 
(at most 3) switches.  Hence, when $\kvec$ is P-stable, we conclude
that the number of encodings is at most polynomially larger than
$|G(\kvec)|$. This proves the ``critical lemma'' for this flow,
and establishes the following.

\begin{theorem} \emph{\cite[Theorem~1.3]{hungarians}}\  
The switch Markov chain
is rapidly mixing on all degree sequences contained in a P-stable
class.
\label{thm:hungarians}
\end{theorem}

Erd\H{o}s et al.\ adapted their analysis to bipartite degree
sequences and directed degree sequences~\cite{hungarians}, 
proving the analogue of
Theorem~\ref{thm:hungarians} in those settings.  
Hence their result extends Theorem~\ref{thm:AK} from strongly stable
to P-stable degree classes, and includes directed degree sequences.
bipartite degree sequences and directed degree sequences.

Applying Theorem~\ref{thm:P-stable}(ii) and 
Theorem~\ref{thm:hungarians} to the degree sequence of
$\G(n,p)$ leads to the following result.

\begin{cor} \emph{\cite[Corollary~8.6]{hungarians}}\
\label{cor:hungarians}
When $n\geq 100$, the degree sequence of
the binomial random graph $\G(n,p)$ satisfies the condition of
Theorem~\ref{thm:P-stable}(ii)
with probability at least $1-3/n$, so long as 
$p, 1-p\geq \frac{5\log n}{n-1}$.
Hence the switch chain is rapidly mixing on $\G(n,p)$
with probability at least $1-3/n$.
\end{cor}

Indeed, applying Theorem~\ref{thm:JS} we can also conclude that if the
conditions of Corollary~\ref{cor:hungarians} hold then with probability
at least $1-3/n$, the
Jerrum--Sinclair chain gives an FPAUS
for sampling from $G(n,p)$. 

\subsubsection{A new notion of stability}

Recently, Gao and Greenhill~\cite{GG} introduced a new notion\footnote{Called ``$k$-stability'' in~\cite{GG}, but here we reserve $k$ for degrees.}
of stability for classes of degree sequences.  

\begin{Def}
\label{def:k-stability}
Given a positive integer $b$ and nonnegative real number $\alpha$,
a graphical degree sequence $\kvec$ is said to be $(b,\alpha)$-\emph{stable}
if $|G(\kvec')|\leq M(\kvec)^\alpha\, |G(\kvec)|$ for every graphical
degree sequence $\kvec'$ with $\| \kvec'-\kvec\|_1\leq b$.
Let $\mathcal{D}_{b,\alpha}$ be the set of all degree sequences
that are $(b,\alpha)$-stable.  A family $\mathcal{D}$ of degree
sequences is \emph{$b$-stable} if there exists a constant
$\alpha>0$ such that $\mathcal{D}\subseteq \mathcal{D}_{b,\alpha}$.
\end{Def}

Gao and Greenhill proved~\cite[Proposition 6.2]{GG} that 2-stability is 
equivalent to P-stability.  The relationship between strong stability
and 2-stability is not known.

Recall that by removing all defect edges, an encoding gives
rise to a graph with degree sequence not too far from $\kvec$.
Gao and Greenhill observe that all degree sequences $\kvec'$ which 
correspond to encodings arising
from the multicommodity flow of~\cite{CDG,GS} satisfy 
$\| \kvec' - \kvec\|_1\leq 8$.  Next, assuming $(8,\alpha)$-stability
they found an upper bound on the number of encodings compatible
with a given graph $Z\in G(\kvec)$: this proves the ``critical lemma''
and leads to the following result.

\begin{theorem} \emph{\cite[Theorem~2.1]{GG}}\ If the
graphical degree sequence $\kvec$ is $(8,\alpha)$-stable for some
$\alpha>0$ then the switch chain on $G(\kvec)$ is rapidly mixing, and
\[ \tau(\varepsilon)\leq 12\, k_{max}^{14}\, n^6\, M^{3+\alpha}\,
    \big( \nfrac{1}{2} M \log M + \log \varepsilon^{-1})\big).\]
\end{theorem}

Gao and Greenhill provided a sufficient condition
for a degree sequence to be 8-stable, and a slightly weaker condition
which guarantees P-stability and strong stability.  These conditions
involve the parameter $J(\kvec)$ defined in (\ref{def:Jd}), and
have been designed to work well for heavy-tailed degree sequences.

\begin{theorem} \emph{\cite[Theorem~2.2]{GG}}\
\label{thm:GG}
\begin{itemize}
\item[\emph{(i)}] Let $\kvec$ be a graphical degree sequence. If $M > 2J(\kvec) + 18 k_{\max} + 56$ then $\kvec$ is (8,8)-stable.
\item[\emph{(ii)}] Suppose that $\mathcal{D}$ is a family of degree sequences
wuch that $M > 2 J(\kvec) + 6 k_{\max} + 2$  for all $\kvec\in\mathcal{D}$.
Then $\mathcal{D}$ is both P-stable and strongly stable.
\end{itemize}
\end{theorem}

The proof of Theorem~\ref{thm:GG}(i) uses the switching method. Then
Theorem~\ref{thm:GG}(ii) follows using the fact, proved 
in~\cite[Lemma 4.1]{GG}, that if every graphical
degree sequence $\kvec'$ with $\| \kvec' - \kvec\|_1 \leq 6$ is
$(2,\alpha)$-stable then $\kvec$ is $(8,4\alpha)$-stable. 

Finally, Gao and Greenhill prove that various families of heavy-tailed degree
sequences satisfy the condition of Theorem~\ref{thm:GG}(i), and
hence are 8-stable, strongly stable and P-stable.  In particular~\cite[Theorem 5.3]{GG},
the family of power-law distribution-bounded sequences with parameter
$\gamma>2$ is P-stable.

\bigskip

Gao and Greenhill gave analogous definitions and results for directed
degree sequences~\cite[Section 7]{GG}.

\subsection{Restricted graph classes}

We briefly describe some related work on using rapidly mixing 
Markov chains to sample from restricted classes of graphs with given degrees.

\subsubsection{Joint degree matrices}  

In some applications, it is desirable to be able to specify not just
the degrees of a graph, but also the number of edges between vertices
with given degrees.  This can help to capture network properties
such as \emph{assortativity}, which is the tendency for vertices
with similar degrees to be adjacent.  A 
\emph{joint degree matrix}~\cite{AGM,PH} 
stores the number of edges $J_{ij}$ with one endvertex of degree $i$ and
the other of degree $j$, for all relevant $i,j$.
A sequential importance sampling approach for sampling graphs with a specified
joint degree matrix was given in the physics
literature by Bassler et al.~\cite{BDGEMT}, without full analysis.

The switch chain can be adapted to sample from the set of all graphs with
a given degree sequence and given joint degree matrix, by rejecting
any transition which would change any entry in the joint degree matrix.
Stanton and Pinar~\cite{StantonPinar} gave empirical evidence that suggests
that the switch chain mixes rapidly on graphs with a prescribed joint
degree matrix, but there are few rigorous results.
Erd\H{o}s, Mikl\'os and Toroczkai~\cite{EMT2015} proved that the switch chain
is rapidly mixing on the set of all \emph{balanced} realizations of a given
joint degree matrix.  Here a realization is balanced if for
all vertices $v$ with degree $i$, the number of
neighbours $w$ of $v$ with degree $j$ is within 1 of the value
$J_{ij}/n_i$, where 
$n_i$ is the number of vertices with degree $i$.
Their proof involved a new Markov chain decomposition 
theorem~\cite[Theorem 4.3]{EMT2015}, similar to that of Martin and 
Randall~\cite{MR}.

Amanatidis and Kleer~\cite{AK} showed that the switch chain is
rapidly mixing on the set of all realizations of any joint degree matrix 
with just two degree classes.
Their analysis is quite technical, and moving beyond two degree classes
seems to be a challenging problem.

\subsubsection{Connected graphs}

The switch chain may disconnect a connected graph, which can be
undesirable in some applications such as communications networks.
One possibility is to simply reject any proposed switch which
would disconnect the graph. Gkantsidis et al.~\cite{GMZ} investigated
the performance of this restricted switch chain empirically, but without 
rigorous analysis.  Note that the set of connected graphs with given degree
sequence was shown to be connected under switches by Taylor~\cite{taylor}
in 1981.

Mahlmann and Schindelhauer~\cite{MS} proposed an alternative operation,
called the $k$-Flipper.  Here a switch is performed if the edges of
the switch are 
at distance at most $k$ apart in the graph.  In the 1-Flipper,
or \emph{flip chain},
the switch operation takes a path of length 3 and exchanges its
endvertices, as shown in Figure~\ref{fig:flip}.  Clearly this operation,
known as a \emph{flip}, cannot disconnect a connected graph.

\begin{figure}[ht!]
\begin{center}
\begin{tikzpicture}[scale=1.0]
\draw [fill] (0,0) circle (0.1); \draw [fill] (0,2) circle (0.1);
\draw [fill] (2,0) circle (0.1); \draw [fill] (2,2) circle (0.1);
\draw [thick,-] (0,0) -- (0,2);  \draw [thick,-]  (2,0) -- (2,2);
\draw [thick,-] (0,0) -- (2,0); 
\draw [<->, line width=3pt] (3.5,1) -- (4.5,1);
\begin{scope}[shift={(6,0)}]
\draw [fill] (0,0) circle (0.1); \draw [fill] (0,2) circle (0.1);
\draw [fill] (2,0) circle (0.1); \draw [fill] (2,2) circle (0.1);
\draw [thick,-] (2,2) -- (0,0) -- (2,0) -- (0,2);  
\end{scope}
\end{tikzpicture}
\caption{Transitions of the flip chain}
\label{fig:flip}
\end{center}
\end{figure}

The flip chain is rapidly mixing on the set of all connected $k$-regular
graphs, for any $k$.  This was investigated by Feder et al.~\cite{FGMS},
with full analysis and improved mixing time bound given by Cooper et al.~\cite{CDGH}.
These proofs involve a comparison argument, where a sequence of
flips is used to simulate a single switch.  If a switch disconnects
or connects components, then a clever ``chaining'' argument from~\cite{FGMS}
is used to stay within the space of connected graphs. 

The flip operation can be used to re-randomise a given connected network
(such as a communications network) without any risk of disconnecting
the network.  Expander graphs are fixed graphs which enjoy
some pseudorandom properties, such as logarithmic diameter and high
connectivity~\cite{HLW}.
Allen-Zhu et al.~\cite[Theorem~4.2]{ABLMO} proved that when 
$k\geq c\log n$ for some positive constant $c$,  
performing $O(k^2 n^2 \log n)$ randomly-chosen flips 
produces an expander with high
probability, starting from any $k$-regular graph.
They also applied their methods to the switch chain, showing
that $O(kn)$ randomly-chosen switches suffice to produce an expander,
with high probability.    Hence in situations where the output
does not need to be close to uniform, but where pseudorandomness
is enough, the runtime of the algorithm can be much shorter.

\section{Conclusion}\label{s:conclusion}

We have discussed rigorously-analysed
algorithms for sampling graphs with a given degree sequence, uniformly or
approximately uniformly.   Some algorithms are inefficient when the maximum
degree becomes too high.  For other approaches, the boundary between
tractable and intractable degree sequences is not clear.   Mapping out
this frontier is an interesting open problem.  Are there families of degree sequences
for which the switch chain is provably slow?
Connections with stability of degree sequences have also been
discussed. As well as their theoretical interest, there are connections
between the stability of degree sequences and network privacy,
as investigated by Salas and Torra~\cite{SalasTorra}.

A challenging open problem is to find an FPRAS for counting graphs
with given (arbitrary) degree sequences. The corresponding problem for
bipartite graphs with solved by Jerrum, Sinclair and Vigoda~\cite{JSV}.

There are many related sampling algorithms which are just outside the
scope of this survey.   
One example is the use of Boltzmann samplers~\cite{DFLS}
to sample from 
other restricted graph classes, including planar graphs~\cite{fusy}.
As well as providing algorithms, this approach can be used to investigate
typical properties of random graphs generated in this way, see for
example~\cite{mcdiarmid}.  

To close, we mention some algorithms
where the degree sequence itself is a random variable.
Some fuzziness in the degree sequence can be useful in some applications,
perhaps to account for inaccuracies in the data, or to avoid overfitting.
The excellent book by Van der Hofstad~\cite{vdH-book} is a very 
good reference for further reading on these topics.  

\begin{itemize}
\item
In the network theory literature, often $k_1,\ldots, k_n$
are i.i.d.\ random variables drawn from some fixed distribution.
If the resulting sum is odd then $k_n$ is increased by 1.
More generally, a degree sequence $\kvec$ can be drawn at random from a given
distribution, and then a graph from $G(\kvec)$ can be sampled
uniformly, or approximately uniformly, using one of the methods discussed
in this survey.
\item 
\emph{Inhomogeneous random graphs}
are similar to the binomial random graph $\mathcal{G}(n,p)$
except that different edges have different probabilities.  For
a sequence $\wvec \in\R^n$ of positive vertex weights
and a function $f:\R^2\to [0,1]$, 
the edge $\{i,j\}$ is included in the graph with probability
$f(w_i,w_j)$,  independently for each edge.
An example is the \emph{generalised random graph model}~\cite{BDM-L} 
with $f(w_i,w_j) = \frac{w_iw_j}{M(\wvec) + w_iw_j}$, where
$M(\wvec) = \sum_{\ell\in [n]} w_{\ell}$ is the sum of the weights.
The Chung-Lu algorithm~\cite{CL} uses $f(w_i,w_j) = w_iw_j/M(\wvec)$,
under the assumption that the maximum entry of $\wvec$ is $o(M(\wvec)^{1/2})$.
The output of the Chung-Lu algorithm is a random graph with expected
degree sequence $\wvec$.
\item 
Another algorithm which produces a graph with degree sequence
close to some target sequence is the 
\emph{erased configuration model}~\cite{BDM-L}.
First sample a uniformly random configuration, and in the corresponding
graph, delete any loops
and delete all but one copy of each multiple edge.  Call the resulting
graph $\widehat{G}$.   
If $\kmax = o(M^{1/2})$ and $R=O(M)$ then, arguing as in the proof of
Theorem~\ref{thm:config}, with high probability only a very small number of
edges were deleted, and hence the degree sequence of $\widehat{G}$
is likely to be very close to $\kvec$.
Other variations of the configuration 
model are described in~\cite[Section 7.8]{vdH-book}, including models 
which are tailored to encourage other network properties, such as clustering.  
\end{itemize}


\thankyou{The author is very grateful to
the anonymous referee and to the following colleagues for their feedback, which
helped improve this survey: 
Martin Dyer, Peter Erd{\H o}s, Serge Gaspers, Pieter Kleer,
Brendan McKay, Eric Vigoda and Nick Wormald. 
This survey is dedicated to Mike. 
}


\myaddress

\bigskip
\bigskip
\bigskip

\appendix
\section{Comment added, 6 August 2025}

After writing this survey paper, I realised that in~\cite[Section~4.1]{FGMS}, Feder, Guetz, Mihail and Saberi's showing that the switch chain is rapidly mixing for any family of degree sequences for which asymptotic enumeration formulae exist for the number of realisations. See in particular~\cite[Corollary 6]{FGMS}. I do not know how I had overlooked this, especially since I had used their ingenious embedding argument for the flip chain analysis from this paper in my own work. I apologise for not discussing~\cite[Section~4.1]{FGMS} in this survey article, and in my other papers on the switch chain for irregular degree sequences.

\end{document}